\documentclass{amsart}
\usepackage[latin1]{inputenc}
\usepackage[english]{babel}
\usepackage{csquotes}

\usepackage{color} 
\usepackage[colorlinks=true,linktoc=all,linkcolor=blue,citecolor=red,filecolor=pink,urlcolor=blue]{hyperref}

\usepackage[hyperref=true,style=alphabetic,citestyle=alphabetic,backend=bibtex,maxnames=100,doi=false,isbn=false,url=false,giveninits=true]{biblatex}

\usepackage{quiver}
\usepackage{enumitem}
\usepackage{graphicx}

\usepackage[mathcal,mathscr]{euscript}

\usepackage{caption}
\usepackage{subcaption}
\usepackage{comment}
\usepackage{import}

\usepackage{amsmath, amssymb, amsfonts,mathabx}
\usepackage{tikz} 
\usetikzlibrary{knots}
\usepackage{pst-knot}
\usepackage{graphicx, overpic}
\usepackage{microtype}
\usepackage{MnSymbol,wasysym}

\usetikzlibrary{decorations.pathreplacing,decorations.markings}
\usetikzlibrary{patterns}
\usetikzlibrary{arrows.meta}
\usepackage{extarrows}
\usetikzlibrary{decorations.markings}

\tikzset{middlearrow/.style={
		decoration={markings,
			mark= at position 0.55 with {\arrow{#1}} ,
		},
		postaction={decorate}
	}
}

\theoremstyle{plain}                    
\newtheorem{theorem}{Theorem}[section]
\newtheorem{lemma}[theorem]{Lemma}
\newtheorem{proposition}[theorem]{Proposition}
\newtheorem{corollary}[theorem]{Corollary}

\newtheorem{problem}[theorem]{Problem}

\newcommand{\theoremnumber}{} 
\newtheorem*{maintheorem}{Theorem \theoremnumber}

\newcommand{\corollarynumber}{} 
\newtheorem*{maincorollary}{Corollary \corollarynumber}

\theoremstyle{definition}
\newtheorem{definition}[theorem]{Definition}
\newtheorem{example}[theorem]{Example}
\newtheorem{remark}[theorem]{Remark}

\numberwithin{equation}{section}

\usepackage{todonotes}

\usepackage[normalem]{ulem}

\newcommand{\nn}{\mathbb N}
\newcommand{\zz}{\mathbb Z}

\newcommand{\ff}{\mathbb F} 

\newcommand{\raag}[1]{A_{#1}} 
\newcommand{\bbg}[1]{BB_{#1}}
\newcommand{\vv}[1]{V(#1)} 
\newcommand{\ee}[1]{E(#1)} 
\newcommand{\lk}[2]{\operatorname{lk}\left( #1,#2 \right)} 
\newcommand{\induced}[2]{#1_{#2}}

\newcommand{\flag}[1]{\Delta_{#1}} 

\newcommand{\suppclique}[3]{K_{#1}(#2,#3)} 
\newcommand{\suppcliqueclique}[2]{K_{#1}(#2)} 
\newcommand{\dist}[3]{\operatorname{d}_{#1}(#2,#3)} 

\newcommand{\dualtree}[1]{{#1}^*}

\newcommand{\order}[1]{\left(\left( #1 \right)\right)}

\newcommand{\bbgm}[1]{O_{#1}}

\newcommand{\shadow}[1]{\Gamma_{#1}}

\makeatletter
\@namedef{subjclassname@2020}{\textup{2020} Mathematics Subject Classification}
\makeatother

\makeatletter
\newcommand{\subsectionnotoc}[1]{%
  \begingroup
  \let\@tocwrite\@gobbletwo
  \subsection*{#1}%
  \leavevmode
  \endgroup
}
\makeatother

\begin{document}

\title[Matroids and isomorphism problems for BBGs]{Matroids and isomorphism problems for Bestvina--Brady groups}

\author{Yu-Chan Chang}
\address{Department of Mathematics, The Ohio State University, 231 W. 18th Ave., Columbus, OH 43210, USA}
\email{chang.2628@osu.edu}

\author{Lorenzo Ruffoni}
\address{Department of Mathematics and Statistics, Binghamton University, Binghamton, NY 13902, USA}
\email{lorenzo.ruffoni2@gmail.com}


\begin{abstract}
We propose a factorization of the graph isomorphism problem for Bestvina--Brady groups (BBGs) through matroid theory.
In particular, we show that finitely presented BBGs depend on their defining graphs only through their cycle matroids.
On the other hand, we construct graphs of arbitrarily high connectivity such that they have non-isomorphic cycle matroids but their BBGs are isomorphic. 
To do so, we
prove that if a  graph   admits a tree clique-spanner, then the Dicks--Leary presentation of its BBG can be explicitly simplified to a right-angled Artin group presentation.
In particular, we show that BBGs defined by dually chordal graphs are right-angled Artin groups.
\end{abstract}

\maketitle

\tableofcontents

\section{Introduction}
In this paper, all graphs and matroids are finite unless otherwise stated.
Let $\Gamma$ be a   graph with vertex and edge sets denoted by $\vv \Gamma$ and $\ee\Gamma$, respectively.
When $\Gamma$ is simplicial, we denote by $\flag\Gamma$ the flag simplicial complex determined by $\Gamma$.
The \emph{right-angled Artin group (RAAG)} $\raag\Gamma$ associated with a simplicial graph $\Gamma$
is the group defined by the following finite presentation:
$$
\raag \Gamma=\big\langle \vv \Gamma \ \big\vert \ [v,w] \ \text{whenever} \ \{v,w\}\in \ee \Gamma\big\rangle.
$$
RAAGs are central objects of study in geometric group theory because they enjoy both a strikingly simple combinatorial definition and an abundance of interesting subgroups; see, for instance, \cite{CH07,KO22}.

Given a simplicial graph $\Gamma$, there is a canonical homomorphism $\chi\colon\raag \Gamma\to \zz$ sending each generator to $1$.
The \textit{Bestvina--Brady group} (BBG) associated with $\Gamma$ is defined to be the kernel of $\chi$: 
$$\bbg \Gamma = \ker (\chi\colon\raag \Gamma\to \zz).$$ 

BBGs were introduced in \cite{BB1997} and have traditionally been used as a source of examples of groups with exotic finiteness properties. 
When $\Gamma$ is connected, Dicks and Leary showed that a (possibly infinite) presentation of $\bbg \Gamma$ can be written down directly from the graph $\Gamma$ itself~\cite{DicksLeary99}.
If $\flag \Gamma$ is simply connected, then the Dicks--Leary presentation can be simplified to a finite one.
There is a growing body of literature devoted to computing algebraic and geometric invariants of BBGs in terms of the combinatorial properties of their defining graphs; see, for instance, \cite{PapadimaSuciuAlgebraicinvariantsforBBGs,DimacaPapadimaSuciuQuasiKahlerBBGs,PapadimaandSuciuBNSRinvariantsandHomologyJumpingLoci,LearySaadetogluTheCohomologyofBBGs,DavisOkun,lorenzo,kochloukovamendonontheBNSRsigmainvariantsoftheBBGs,ChangJSJofBBGs,DR22,BL25,CRKR25,SplittingBBG25, CR26,DehnFunc26,jialin}.
By \cite[Corollary 3.10]{CR26}, every RAAG is a BBG, so the study of BBGs can be regarded as a generalization of the study of RAAGs.
In the present work, we contribute to the literature by considering the following problems for BBGs.

\begin{problem}[Graph isomorphism problem for BBGs]\label{problem graph}
    Which  simplicial graphs define isomorphic BBGs?
\end{problem}
 
\begin{problem}[RAAG recognition problem for BBGs]\label{problem raag}
    Which  simplicial graphs define BBGs that are isomorphic to RAAGs?
\end{problem}

In this paper, we provide a factorization of Problem~\ref{problem graph} via matroid theory and a solution to Problem~\ref{problem raag} for dually chordal graphs, extending our previous work in \cite{CR26}. 
These two problems are related through a rigidity result of Droms relating RAAGs to their defining graphs~\cite{DromsIsomorphismsofGraphGroups}. 
We now provide some background on Problems~\ref{problem graph} and~\ref{problem raag} and describe our contributions in detail.
Readers interested only in RAAG recognition may safely skip the treatment of matroids in  \S\ref{sec:intro isom} and \S\ref{sec:matroids}.

\subsection{Graph isomorphism problem for BBGs}\label{sec:intro isom}

Droms showed that two RAAGs are isomorphic if and only if their defining graphs are isomorphic~\cite{DromsIsomorphismsofGraphGroups}.
This is perhaps one of the main features of RAAGs, as it establishes a connection between their algebraic properties and the combinatorial properties of their defining graphs.

On the other hand, BBGs are less rigid than RAAGs. For instance, the BBG defined by any tree with $n$ edges is isomorphic to $\ff_n$.
Similarly, the BBGs defined by the graphs in Figure~\ref{fig:noniso_2connected} are both isomorphic to the RAAG defined by the path on five vertices; see also~\cite[Example 3.7]{CR26}. The graphs in Figure~\ref{fig:noniso_2connected} are \textit{biconnected},
(that is, they are connected and have no cut vertices),
and one can be obtained from the other by flipping along the separating edge $\{u,v\}$.
This leads to Problem~\ref{problem graph}, which asks for a combinatorial characterization of the class of  simplicial graphs that define a given BBG.

\begin{figure}[h]
    \centering
    \begin{tikzpicture}[scale=0.6]
\draw [thick] (2,0)--(0,0);
\draw [thick] (2,0)--(2,2);
\draw [thick] (2,2)--(0,2);
\draw [thick] (0,0)--(0,2);
\draw [thick] (2,0)--(0,2);

\draw [thick] (2,0)--(4,0);
\draw [thick] (4,0)--(4,2);
\draw [thick] (4,2)--(2,2);
\draw [thick] (2,0)--(4,2);

\draw [fill] (0,0) circle [radius=0.1];
\draw [fill] (2,0) circle [radius=0.1];
\draw [fill] (2,2) circle [radius=0.1];
\draw [fill] (0,2) circle [radius=0.1];
\draw [fill] (4,0) circle [radius=0.1];
\draw [fill] (4,2) circle [radius=0.1];

\node [above] at (2,2) {$u$};
\node [below] at (2,0) {$v$};

\begin{scope}[shift={(7,0)}]

\draw [thick] (2,0)--(0,0);
\draw [thick] (2,0)--(2,2);
\draw [thick] (2,2)--(0,2);
\draw [thick] (0,0)--(0,2);
\draw [thick] (0,2)--(2,0);

\draw [thick] (2,0)--(4,0);
\draw [thick] (4,0)--(4,2);
\draw [thick] (4,2)--(2,2);
\draw [thick] (2,2)--(2,0);
\draw [thick] (2,2)--(4,0);

\draw [fill] (0,0) circle [radius=0.1];
\draw [fill] (2,0) circle [radius=0.1];
\draw [fill] (2,2) circle [radius=0.1];
\draw [fill] (0,2) circle [radius=0.1];
\draw [fill] (4,0) circle [radius=0.1];
\draw [fill] (4,2) circle [radius=0.1];

\node [above] at (2,2) {$u$};
\node [below] at (2,0) {$v$};
\end{scope}
\end{tikzpicture}
    \caption{Non-isomorphic biconnected graphs whose associated BBGs are isomorphic.}
    \label{fig:noniso_2connected}
\end{figure}
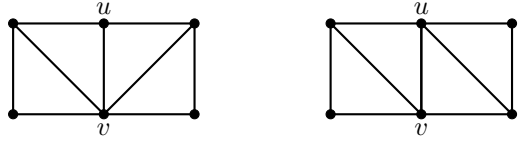

While the BBG defined by a tree is a free group, the presence of cycles in a graph gives rise to more complicated groups.
For example, if $\Gamma=C_n$ is a cycle of length $n\geq 4$, then $\bbg{C_n}$ is finitely generated but not finitely presented.
More generally, the topology of the flag complex $\flag \Gamma$ is reflected in the finiteness properties of  $\bbg \Gamma$; see \cite{BB1997}.

The starting point for this paper was the observation that the presentation for BBGs obtained by Dicks and Leary in \cite{DicksLeary99} seems to only depend on the structure of cycles in $\Gamma$ and not on $\Gamma$ itself.
More precisely, the generators are the oriented edges of $\Gamma$, and each cycle gives infinitely many relators; see  \S\ref{sec: presentation BBG} for details.

The structure of cycles of a graph is conveniently encoded in its so-called \textit{cycle matroid}. 
For example, the graphs in Figure~\ref{fig:noniso_2connected} have isomorphic cycle matroids.
More generally, a \textit{matroid} is a combinatorial structure consisting of a finite set $E$, called the \textit{ground set}, and a collection of subsets of $E$, called \textit{circuits}, 
such that this collection satisfies certain axioms generalizing the structure of cycles in a  graph. We refer the reder to~\cite{Oxley} for an introduction to matroids.
It is tempting to approach Problem~\ref{problem graph} through the corresponding isomorphism problem for matroids.

Since a circuit is an unordered subset of $E$, but a relator is a word in $E$, a choice of a cyclic ordering of the circuit is needed to write down the corresponding relator in a group presentation.
Matroids whose circuits can be consistently ordered are called \textit{orderable}; this notion was introduced by Crenshaw and Oxley~\cite{oxleycrenshaw}.
Since groups have inverses, we also need to equip the matroid with an \textit{orientation}. 
We therefore adapt the notion of orderability from matroids to \textit{oriented matroids}; see \cite{orientedmatroids, Laura2025} for an introduction to oriented matroids.

In \S\ref{sec:matroids}, we introduce the notion of an \textit{orderly matroid} $\mathfrak M$ and associate with it a group called the \textit{orderly matroid group} (OMG), denoted by $\bbgm{\mathfrak M}$; see Definition~\ref{def: orderly matroid} and Definition~\ref{def:OMG}, respectively.
Roughly speaking, an oriented matroid is orderly if it admits an ordering of its circuits that is compatible with its orientation. For example, the cycle matroid of a graph $\Gamma$ is orderly; we call it the \textit{orderly cycle matroid} of $\Gamma$ and denote it by $\mathfrak M_\Gamma$. 
The associated OMG is then defined by a group presentation analogous to the Dicks--Leary presentation for BBGs from \cite{DicksLeary99}:
the generators are the elements of the ground set $E$, and each circuit gives rise to a family of relators.

Some matroids are orderable but not orderly. For example, the uniform matroids $U_{2,n}$ for $n\geq 4$. Notice that $U_{2,n}$ is not \textit{graphic} for every $n\geq 4$, that is, it is not isomorphic to the cycle matroid of any graph.
We are not aware of an example of an orderly matroid that is not a graphic matroid.
Similarly, it is an open question whether a 3-connected orderable binary matroid is graphic (see \cite[Conjecture~4]{oxleycrenshaw}).
It is worth noting that
the notion of orderability does not pass to minors (see \cite[Example 15]{oxleycrenshaw}).

The next statement is our first main result. 
Note that RAAGs and BBGs can be defined for both simplicial and non-simplicial graphs; see Remark~\ref{remark:non simplicial RAAG BBG}.

\begin{theorem}\label{mainthm matroid intro}
Let $\Gamma$ be a connected graph,
$\mathfrak M_\Gamma$ its orderly cycle matroid, and  $\bbg \Gamma$ the BBG of $\Gamma$.
Let  $\bbgm{\mathfrak M}$ be the OMG of an orderly matroid $\mathfrak M$.
Then $\bbgm{\mathfrak M_\Gamma}\cong \bbg \Gamma$. That is, the following diagram commutes.
\vspace{-3em}
\begin{figure}[ht!]
\centering
\begin{tikzpicture}[mapsto/.style={|->},every node/.style={inner sep=2pt}]
\matrix(m)[matrix of math nodes,column sep=2.5em,row sep=1.8em] 
{
& & & & \\
& \left\{
\begin{gathered}
\text{connected}\\
\text{graphs}
\end{gathered}
\right\} & & \left\{
\begin{gathered}
\text{finitely}\\
\text{generated}\\
\text{groups}
\end{gathered}
\right\} & \\
& & & & \\
& & \left\{
\begin{gathered}
\text{orderly matroids}\\
\text{up to orientation}
\end{gathered}
\right\} & & \\
};
\draw[->] (m-2-2) -- node[midway,above] {$\Gamma\mapsto\bbg\Gamma$} (m-2-4);

\draw[->] (m-2-2) -- node[midway,below left] {$\Gamma\mapsto\mathfrak M_\Gamma$} (m-4-3);

\draw[->] (m-4-3) -- node[midway,below right] {$\mathfrak M\mapsto\bbgm{\mathfrak M}$} (m-2-4);
\end{tikzpicture}
\label{commu diag}
\end{figure}
\end{theorem}

In other words, the map  that sends a graph to the associated BBG factors through the orderly matroid of that graph.
This provides a factorization of the graph isomorphism problem for BBGs (Problem~\ref{problem graph}) through the corresponding isomorphism problem for orderly matroids and their OMGs.

One might hope that this factorization is complete, in the sense that, at least for finitely presented BBGs, two graphs define isomorphic BBGs if and only if their cycle matroids are isomorphic.
The following result establishes that this is not the case; that is, the map that sends an orderly matroid to its OMG fails to be injective in a quite structural way. 
A  graph is called \emph{$k$-connected} if it has more than $k$ vertices, and removing fewer than $k$ vertices does not disconnect the graph.

\begin{theorem}\label{introthm non isom}
    For every $k\geq 1$, there are $k$-connected chordal graphs $\Gamma$ and $\Lambda$ such that the following statements hold. 
    
    \begin{enumerate}
        \item $\Gamma\not \cong \Lambda$ (so in particular, $\mathfrak M_\Gamma \not \cong \mathfrak M_\Lambda$ when $k\geq 3$). 
        \item $\bbg \Gamma \cong \bbg \Lambda$ (or equivalently, $\bbgm{\mathfrak M_\Gamma} \cong \bbgm{\mathfrak M_\Lambda}$).
    \end{enumerate}
\end{theorem}

The key to proving Theorem~\ref{introthm non isom} is that the graphs $\Gamma$ and $\Lambda$ are constructed specifically so that $\bbg \Gamma$ and  $\bbg \Lambda$ are both isomorphic to the same RAAG; see \S\ref{sec:intro raag}.

After defining OMGs and proving that the above diagram commutes, we investigate some of their structural properties. 
A subtle point is that the presentation of the OMG associated to an orderly matroid depends a priori on both the orientation of the matroid and the ordering of the circuits.
We address this in detail.

First of all, the OMG is completely insensitive to the reorientation; see Lemma~\ref{lem:reorientation}.
Its dependence on the ordering is more delicate.
For cycle matroids arising from simplicial graphs with simply connected flag complex, the choice of ordering is irrelevant; see Corollary~\ref{cor: s.c. OMG}. However, this is not true in general; see Remark~\ref{jialin}. This raises the question of how different orderings of the same matroids are related.

In \S\ref{sec: orderly graph}, we show how to use the ordering of an orderly matroid $\mathfrak M$ to define the so-called \textit{orderly graph} $\shadow{\mathfrak M}$, which provides a ``shadow'' of $\mathfrak M$.
For the orderly cycle matroid $\mathfrak M_\Gamma$ of a graph $\Gamma$, this construction recovers $\Gamma$, that is, $\shadow{\mathfrak M}\cong \Gamma$; see Lemma~\ref{lem:orderly graph of a graph}.
More generally, we show that the OMG of an orderly matroid $\mathfrak M$ surjects onto the BBG defined by a graph constructed from $\shadow{\mathfrak M}$; see Proposition~\ref{prop:AOMG=RAAG}.
In particular, the following result states that the OMG of an orderly graphic matroid is isomorphic to the corresponding BBG.

\begin{theorem}\label{thm:faithful}
Let $\mathfrak M=(E,\mathcal{C},\omega)$ be an orderly graphic matroid. 
Let $\mathfrak M_1,\dots,\mathfrak M_n$ be the connected components of $\mathfrak M$, and let $\shadow{\mathfrak M_i}$ be the orderly graph of $\mathfrak M_i$. 
    Then  $$\bbgm{\mathfrak M} \cong  \Asterisk_{i=1}^n \bbg{\shadow{\mathfrak{M}_i}}\cong \bbg{\bigvee_{i=1}^n \shadow{\mathfrak M_i}}.$$
\end{theorem}

A graphic matroid can be isomorphic to the cycle matroid of many non-isomorphic graphs.
The point of Theorem~\ref{thm:faithful}  is that a specific graph is reconstructed canonically just from the ordering $\omega$ of $\mathfrak M$.
Indeed, we show that, for a graphic matroid, there is a bijection between the set of  defining graphs and a certain set of orderings; see Proposition~\ref{prop:defining graphs and orderings}.

\subsection{RAAG recognition problem for BBGs}\label{sec:intro raag}
Now, let $\Gamma$ be a simplicial graph.
As mentioned above, the flag complex $\flag \Gamma$ is simply connected if and only if $\bbg \Gamma$ is finitely presented; see \cite{BB1997}.
In this case, Papadima and Suciu showed that, given a spanning tree $T$ of $\Gamma$, one can obtain a finite presentation for $\bbg \Gamma$ whose generators are the edges of $T$ and whose relators are commutators, not necessarily commutators of generators; see \cite[Corollary 2.3]{PapadimaSuciuAlgebraicinvariantsforBBGs}.
Problem~\ref{problem raag} asks for conditions on a graph ensuring that the Papadima--Suciu presentation of a BBG can be simplified to a RAAG presentation, in which every relator is a commutator between generators.
It is worth noting that Bridson showed in \cite{Bridson2020}  that determining whether a group presented by commutators is a RAAG is an undecidable problem. Moreover, there are BBGs not isomorphic to any RAAG; see \cite[Example 2.8]{PapadimaSuciuAlgebraicinvariantsforBBGs}. 

In \cite[Theorem B]{CR26}, we showed that the Papadima--Suciu presentation associated with a spanning tree $T$ of $\Gamma$ can be simplified to a RAAG presentation whenever $T$ is a \textit{tree 2-spanner}.
This means that the distances in $T$ are at most twice the distances in $\Gamma$: for any $u,v\in \vv \Gamma$, we have $\dist T uv\leq 2 \dist \Gamma u v$; see the left-hand picture in Figure~\ref{fig:TCS} for an example. 

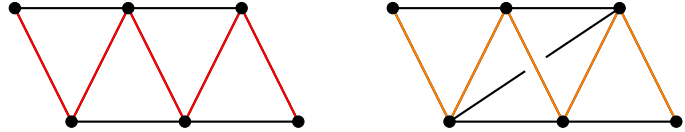
\begin{figure}[ht!]
    \centering
    \begin{tikzpicture}[scale=0.5]

\draw [thick] (0,3)--(3,3)--(1.5,0)--(0,3);
\draw [thick] (1.5,0)--(4.5,0)--(3,3);
\draw [thick] (3,3)--(6,3)--(4.5,0);
\draw [thick] (4.5,0)--(7.5,0)--(6,3);
\draw [thick, red] (0,3)--(1.5,0)--(3,3)--(4.5,0)--(6,3)--(7.5,0);
\draw [fill] (0,3) circle [radius=0.15];
\draw [fill] (1.5,0) circle [radius=0.15];
\draw [fill] (3,3) circle [radius=0.15];
\draw [fill] (4.5,0) circle [radius=0.15];
\draw [fill] (6,3) circle [radius=0.15];
\draw [fill] (7.5,0) circle [radius=0.15];

\begin{scope}[shift={(10,0)}]

\draw [thick] (0,3)--(3,3)--(1.5,0)--(0,3);
\draw [thick] (1.5,0)--(4.5,0)--(3,3);
\draw [thick] (3,3)--(6,3)--(4.5,0);
\draw [thick] (4.5,0)--(7.5,0)--(6,3);
\draw [domain=1.5:3.5, smooth, variable=\x, thick] plot ({\x}, {2*\x/3-1});
\draw [domain=4.05:6, smooth, variable=\x, thick] plot ({\x}, {2*\x/3-1});
\draw [thick, orange] (0,3)--(1.5,0)--(3,3)--(4.5,0)--(6,3)--(7.5,0);
\draw [fill] (0,3) circle [radius=0.15];
\draw [fill] (1.5,0) circle [radius=0.15];
\draw [fill] (3,3) circle [radius=0.15];
\draw [fill] (4.5,0) circle [radius=0.15];
\draw [fill] (6,3) circle [radius=0.15];
\draw [fill] (7.5,0) circle [radius=0.15];
\end{scope}
\end{tikzpicture}
    \caption{The spanning tree (red) in the left-hand picture is a tree $2$-spanner. The spanning tree (orange) in the right-hand picture is a tree clique-spanner but not a tree $2$-spanner.}
    \label{fig:TCS}
\end{figure}

We generalize our previous work~\cite{CR26} to a much wider class of spanning trees, which provides a new criterion for Problem~\ref{problem raag}.

\begin{theorem}\label{mainthm:raag recognition}
Let $\Gamma$ be a  simplicial graph. 
If $\Gamma$ admits a tree clique-spanner $T$, then $\bbg\Gamma$ is a RAAG.
More precisely, we have $\bbg \Gamma\cong\raag{\dualtree T}$, where $\dualtree T$ is the dual graph of $T$ in $\Gamma$. 
\end{theorem}

Here, a \textit{tree clique-spanner} is a spanning tree $T$ of $\Gamma$ such that, for every edge $\{u,v\}\in \ee \Gamma$, the unique path from $u$ to $v$ in $T$ is contained in a clique $K$ of $\Gamma$; see the right-hand picture in Figure~\ref{fig:TCS} for an example.
The \textit{dual graph} $\dualtree T$ is the graph whose vertex set is $\ee T$, and two vertices are adjacent whenever the corresponding edges of $T$ lie in a common clique of $\Gamma$.

By \cite[Theorem B]{CR26} and Theorem~\ref{mainthm:raag recognition}, the BBG defined by each of the graphs in Figure~\ref{fig:TCS} is a RAAG. Moreover,~\cite[Lemma 3.1]{CR26} implies that tree $2$-spanners are tree clique-spanners. Thus, Theorem~\ref{mainthm:raag recognition} recovers \cite[Theorem B]{CR26}. 
However, there are graphs that admit tree clique-spanners but no tree $2$-spanners, such as
the right-hand graph in Figure~\ref{fig:TCS} and the graph $\Gamma$ in Theorem~\ref{introthm non isom}.
Therefore, Theorem~\ref{mainthm:raag recognition} is a genuine generalization of~\cite[Theorem B]{CR26}.

Tree clique-spanners have been studied in the literature under a variety of names. 
A graph is \textit{chordal} if every induced cycle\footnote{In this paper, a \textit{cycle}  in a graph $\Gamma$ is a subgraph $C$ such that every vertex of $C$ has two neighbors in $C$.} of length at least four has a chord. A graph is \textit{dually chordal} if it is the intersection graph of the maximal cliques of a chordal graph.
It was shown in~\cite{SB94} that a graph admits a tree clique-spanner if and only if it is dually chordal.
By \cite[Corollary 1]{BDCV98}, whether a graph $\Gamma$ is dually chordal can be recognized in linear time $O(|\vv\Gamma|+|\ee\Gamma|)$.
Chordality of $\Gamma$ is known to be equivalent to the coherence of $\raag{\Gamma}$ and $\bbg\Gamma$; see \cite{dromsraag3manifolds,BL25}.
Unlike chordal graphs, full subgraphs of dually chordal graphs may not be dually chordal. 
A \textit{strongly chordal graph} is a chordal graph such that every even cycle of length at least six has an odd chord. Equivalently, a graph is strongly chordal if and only if it is a dually chordal graph and every full subgraph is dually chordal; see \cite[Corollary 3]{BDCV98}. 
Theorem~\ref{mainthm:raag recognition} has the following corollary.
In particular, property \eqref{item:intro new BBG minsep} provides an easy and flexible way to construct graphs with tree clique-spanners. We use this property in the proof of Theorem~\ref{introthm non isom}.

\begin{corollary}\label{cor:more bbgs}
    Let $\Gamma$ be a  connected simplicial  graph satisfying one of the following properties:
    \begin{enumerate}
        \item $\Gamma$ is dually chordal.
        \item $\Gamma$ is strongly chordal.
        \item \label{item:intro new BBG minsep} The minimal separators of $\Gamma$ are connected and pairwise disjoint.
    \end{enumerate}
    Then $\Gamma$ admits a tree clique-spanner, and $\bbg \Gamma$ is a RAAG.
\end{corollary}

Beyond the group-theoretic applications, we prove some structural properties of tree clique-spanners that may be of independent interest.
First, we show that if a graph admits a tree clique-spanner, then its flag complex is contractible; see Corollary~\ref{cor:contractible}.
Second, we show in Proposition~\ref{prop:loc_con} that every tree clique-spanner of a biconnected  graph is \textit{locally connected}: for each vertex of the graph, its link in the tree clique-spanner induces a connected subgraph of the ambient graph. This extends Cai's work on tree 2-spanners of 2-connected graphs; see~\cite[Corollary~1.2]{Cai1997}.
Finally, we show that every dually chordal graph admits a locally connected spanning tree; see Corollary~\ref{cor:loc con dually chordal}. This extends a result on biconnected strongly chordal graphs in \cite[Corollary 5]{LCC07}.

\subsectionnotoc{Outline of the paper.}
The paper is organized as follows. In \S\ref{sec: presentation BBG}, we review the Dicks--Leary presentation for BBGs.
In \S\ref{sec:matroids}, we develop the theory of orderly matroids and their orderly matroid groups (OMGs). We prove Theorem~\ref{mainthm matroid intro}, which provides a factorization of the graph isomorphism problem for BBGs (Problem~\ref{problem graph}) through matroid theory, and Theorem~\ref{thm:faithful}, which shows that the OMG of an orderly graphic matroid is isomorphic to the BBG on a graph that can be constructed from the ordering of the matroid.
In \S\ref{sec:tcs}, we introduce tree-clique spanners and discuss their relation to dually chordal graphs.
Finally, in \S\ref{sec:isoproblems}, we prove Theorem~\ref{mainthm:raag recognition}, which 
addresses the RAAG recognition problem (Problem~\ref{problem raag}), and Theorem~\ref{introthm non isom}, which shows the limitations of the matroid-theoretic approach to Problem~\ref{problem graph}.

\subsectionnotoc{Notations.}
We collect some notations for the reader's convenience.
\begin{itemize}
    \item $\Gamma$ denotes a graph.
    \item $\raag \Gamma$ and $\bbg \Gamma$ denote the RAAG and BBG defined by $\Gamma$, respectively.
    \item $M=(E,\mathcal C)$ denotes a matroid with ground set $E$, and $\mathcal C$ denotes the set of circuits.
    \item $\mathfrak M=(E,\mathcal C,\omega)$ denotes an orderly matroid, and $\omega$ denotes the compatible ordering of $\mathfrak M$.
    \item $\bbgm{\mathfrak M}$ denotes the orderly matroid group (OMG) of the orderly matroid $\mathfrak M$.  
    \item $\shadow{\mathfrak M}$ denotes the orderly graph of the orderly matroid $\mathfrak M$. 
    \item $\dualtree T$ denotes the dual graph of a spanning tree $T$.
\end{itemize}

\subsectionnotoc{Acknowledgments.}
L.R. was partially supported by INDAM-GNSAGA.  
We thank Jialin Lei for sharing his preprint that led to Remark~\ref{jialin} and 
Matt Zaremsky for his helpful suggestions. Part of this work was completed while the first author was a postdoc at Wesleyan University; we thank the Department of Mathematics for its hospitality.

\section{Presentations for BBGs}\label{sec: presentation BBG}
Let $\Gamma$ be a simplicial graph. 
Let $\raag \Gamma$ and $\bbg \Gamma$ be the RAAG and BBG defined by $\Gamma$, respectively; see the Introduction for definition.
We now review the presentation for BBGs obtained by Dicks and Leary in \cite{DicksLeary99}.

Suppose that $\Gamma$ is connected, so $\bbg \Gamma$ is finitely generated; see \cite{BB1997}. 
Fix, once and for all, an arbitrary orientation for the edges of $\Gamma$.
The set of oriented edges is denoted by $\ee \Gamma ^\pm$. 
For each $e\in \ee \Gamma ^\pm$, we denote by $\bar e$ the edge with the opposite orientation.
The Dicks--Leary presentation for $\bbg\Gamma$ is the following infinite presentation: 
\begin{equation}\label{eq:infdicksleary}
    \bbg\Gamma=\langle \ee \Gamma ^\pm \mid I_\Gamma,R_\Gamma \rangle,
\end{equation}
where $I_\Gamma$ and $R_\Gamma$ are the sets of relators defined as follows:

\begin{itemize}
    \item $I_\Gamma$ consists of the relators of the form $e\bar e$ for each $e\in\ee\Gamma^\pm$.

    \item $R_\Gamma$ consists of the relators of the form 
    $e_1^k\cdots e_l^k$ for every $k\in \zz$ and for every oriented cycle $(e_1,\dots,e_l)$ with $e_i\in \ee \Gamma^\pm$.
\end{itemize}

Note that reversing the orientation of some edges results in replacing the corresponding generators with their inverses and hence gives an equivalent presentation.
Moreover, when the flag complex $\flag \Gamma$ is simply connected, the infinite presentation~\eqref{eq:infdicksleary} can be reduced to the following finite presentation; see \cite[Corollary~3]{DicksLeary99}:
\begin{equation}\label{eq:dicksleary}
    \bbg\Gamma=\langle \ee \Gamma ^\pm \mid I_\Gamma,R^\Delta_\Gamma \rangle,
\end{equation}
where $I_\Gamma$ is defined as above, and $R^\Delta_\Gamma$ consists of the two relators $e_1e_2e_3$ and $e_3e_2e_1$ for each oriented triangle $\tau=(e_1,e_2,e_3)$ of $\Gamma$ shown in Figure~\ref{fig: triangle relator}.

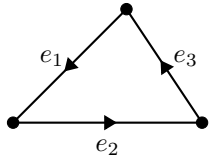
\begin{figure}[ht!]
\centering\usetikzlibrary{decorations.markings, arrows.meta}

\tikzset{
    mid arrow/.style={
        postaction={
            decorate,
            decoration={
                markings,
                mark=at position 0.55 with {\arrow[scale=1]{Triangle}}
            }
        }
    }
}

\begin{tikzpicture}[scale=0.5]
\draw[thick, mid arrow] (3,3) -- (0,0);
\draw[thick, mid arrow] (0,0) -- (5,0);
\draw[thick, mid arrow] (5,0) -- (3,3);
\draw [fill] (0,0) circle [radius=0.15];
\draw [fill] (5,0) circle [radius=0.15];
\draw [fill] (3,3) circle [radius=0.15];

\node [above left] at (1.6,1.25) {$e_1$};
\node [below] at (2.5,-0.2) {$e_2$};
\node [above right] at (4,1.25) {$e_3$};
\end{tikzpicture}
    \caption{An oriented triangle.}
    \label{fig: triangle relator}
\end{figure}

From the presentation~\eqref{eq:dicksleary}, one can deduce a solution to a particular case of the graph isomorphism problem for BBGs (Problem~\ref{problem graph}).

\begin{proposition}\label{free}
A graph $\Gamma$ is a tree with $n$ edges if and only if $\bbg \Gamma =\ff_n$.
\end{proposition}
\begin{proof}
    If $\Gamma$ is a tree with $n$ edges, then it has no cycles. Thus, the Dicks--Leary presentation is a presentation of $\ff_n$.
    Conversely, suppose $\bbg \Gamma =\ff_n$. Then $\bbg \Gamma$ does not contain $\zz^2$ subgroups, so $\Gamma$ is triangle-free. Since $\bbg \Gamma$ is finitely presented, the flag complex $\flag \Gamma$ is simply connected by \cite{BB1997}. Hence, the graph $\Gamma$ is a tree, necessarily with $n$ edges, by the first part of this proof.
\end{proof}

When $\Gamma$ is connected, one sees from the presentation~\eqref{eq:infdicksleary} that $\bbg \Gamma$ is the free product of the BBGs defined by the biconnected components of $\Gamma$.
Therefore, it is natural to restrict our attention to BBGs associated with biconnected graphs.
Similarly, if $e$ is a separating edge of $\Gamma$, then $\bbg \Gamma$ splits over $\zz =\langle e\rangle$; see \cite{lorenzo,ChangJSJofBBGs, SplittingBBG25} for more results about splittings of BBGs.

We now describe the effect of a more subtle graph operation on the associated BBGs.
Let $\Lambda_1$ and $\Lambda_2$ be  simplicial graphs, and choose vertices $u_i,v_i\in \vv{\Lambda_i}$ for $i=1,2$.
Let $\Gamma$ be the  simplicial graph obtained from $\Lambda_1$ and $\Lambda_2$ by identifying $u_1$ with $u_2$ and $v_1$ with $v_2$, and by identifying any resulting parallel edges. We denote by $\{u,v\}$ the set of vertices of $\Gamma$ arising from $\{u_i,v_i\}$.
The graph $\Gamma'$ obtained from $\Lambda_1$ and $\Lambda_2$ by identifying $u_1$ with $v_2$ and $v_1$ with $u_2$, and by identifying any parallel edges if needed. Then $\Gamma'$ is said to be obtained from $\Gamma$ by a \textit{Whitney twist} about $\{u,v\}$.
The graphs $\Lambda_1$ and $\Lambda_2$ are called the \textit{pieces} of the Whitney twist.

Two  biconnected simplicial graphs are called \textit{$2$-isomorphic} if one can be obtained from the other by a sequence of  Whitney twists. For example, the graphs in Figure~\ref{fig:noniso_2connected} are $2$-isomorphic via one Whitney twist about $\{u,v\}$.

\begin{proposition}\label{prop:2isobbg}
   Let $\Gamma$ and $\Gamma'$ be 
    biconnected simplicial graphs such that $\flag\Gamma$ and $\flag{\Gamma'}$ are simply connected. If $\Gamma$ and $\Gamma'$ are $2$-isomorphic, then $\bbg{\Gamma}\cong\bbg{\Gamma'}$.
\end{proposition}

\begin{proof}
    Without loss of generality, suppose that  $\Gamma'$ is obtained from $\Gamma$ by a Whitney twist about a pair of vertices $\{u,v\}$.
    Since $\flag\Gamma$ and $\flag{\Gamma'}$ are simply connected, the vertices $u$ and $v$ are adjacent, and $e=\{u,v\}$ is a separating edge of both $\Gamma$ and $\Gamma'$. 
    
    Let $\Lambda_1$ and $\Lambda_2$ be the pieces of the Whitney twist.
    For $i=1,2$, let $e_i$ be the edge of $\Lambda_i$ corresponding to $e$, and let $f_{i1}, \dots, f_{in_i}$ be the remaining edges of $\Lambda_i$.
    After fixing an orientation of the edges, we have the Dicks--Leary presentation \eqref{eq:infdicksleary} for $\bbg{\Lambda_i}$: 
    $$\bbg{\Lambda_i} = \langle  e_i, \bar e_i , f_{ij}, \bar f_{ij}, j=1,\dots, n_i\mid I_{\Lambda_i},R_{\Lambda_i}\rangle.$$

    Since $e$ is a separating edge of both $\Gamma$ and $\Gamma'$, the groups $\bbg \Gamma$ and $\bbg{\Gamma'}$ split over the cyclic subgroup $\langle e\rangle$ .
    These splittings give the presentations
    $$\bbg \Gamma = \langle  e_i,\bar e_i , f_{ij},\bar f_{ij}, i=1,2, j=1,\dots, n_i \mid e_1=e_2,I_{\Lambda_i},R_{\Lambda_i},i=1,2\rangle $$
    and
    $$\bbg {\Gamma'} = \langle  e_i,\bar e_i , f_{ij},\bar f_{ij}, i=1,2, j=1,\dots, n_i \mid e_1= \bar  e_2,I_{\Lambda_i},R_{\Lambda_i},i=1,2\rangle.
    $$
(These are not the Dicks--Leary presentations of these BBGs, as they do not include relators  associated with cycles that visit both pieces of the Whitney twist.)

Define a map on generators by $e_1\mapsto e_1$, $e_2\mapsto \bar e_2= e^{-1}_2$, $f_{1j}\mapsto f_{1j}$, and $f_{2j}\mapsto \bar f_{2j} = x^{-1}_{2j}$. This map
extends uniquely to a homomorphism $\psi\colon\bbg\Gamma \to \bbg {\Gamma'}$.
We only need to check that $\psi$ preserves the relators in $R_{\Lambda_2}$.
A relator in $R_{\Lambda_2}$ is of the form $y_1^k\cdots y_p^k$, where $k\in \zz$ and $y_s\in \{e_2,\bar e_2, f_{2j},\bar f_{2j} \mid j=1,\dots,n_2\}$ for $s=1,\dots,p$.
Then $\psi(y^k_1\cdots y^k_p)=y^{-k}_1\cdots y^{-k}_p$ is also a relator in $R_{\Lambda_2}$.
The map $\psi$ admits an inverse defined in the analogous way.
Therefore, the map $\psi\colon\bbg\Gamma \to \bbg {\Gamma'}$ is an isomorphism. 
\end{proof}

\begin{remark}\label{jialin}
The assumption in Proposition~\ref{prop:2isobbg} that the flag complexes are simply connected is necessary.
Jialin Lei has constructed an example of two biconnected graphs that are $2$-isomorphic but have non-isomorphic BBGs.
These BBGs are not finitely presented; see \cite{jialin}.
\end{remark}

Whitney twists have been studied in graph theory and matroid theory. Whitney's 2-isomorphism theorem states that two biconnected graphs are 2-isomorphic if and only if their cycle matroids are isomorphic; see \cite[Theorem 5.3.1]{Oxley}.
In~\S\ref{sec:matroids}, we introduce a BBG-like group associated with any oriented matroid whose circuits satisfy a suitable orderability condition.
(The reader only interested in  RAAGs can safely skip \S\ref{sec:matroids}.)
In~\S\ref{sec:isoproblems}, we construct non-isomorphic graphs with non-isomorphic cycle matroids but isomorphic BBGs. This construction relies on a new criterion for a BBG to be a RAAG based on a certain type of spanning trees that we study in \S\ref{sec:tcs}.
 
In the following, it will be convenient to consider RAAGs and BBGs defined by graphs that are not necessarily simplicial, in the sense of the next remark.

\begin{remark}\label{remark:non simplicial RAAG BBG}
    One can define a RAAG on a non-simplicial graph using the usual presentation.
    The resulting group is the same RAAG defined by the simplicial graph obtained from the original graph by removing all loops and identifying parallel edges.
    One can then define the BBG as the kernel of the homomorphism sending all generators of the RAAG to $1$.
    Therefore, allowing non-simplicial graphs does not produce new RAAGs or BBGs.
    This will make it more convenient to state some theorems in this paper.
\end{remark}

\section{Groups and graphs associated to  matroids}\label{sec:matroids}
In this section, we introduce a certain class of matroids and associate a finitely generated group to each such matroid.
A \emph{matroid} $M=(E,\mathcal C)$ is a combinatorial object consisting of a finite set $E$, called the \textit{ground set}, and a collection $\mathcal C \subseteq \mathcal P(E)$ of subsets of $E$, called the \textit{circuits}. These circuits satisfy certain axioms that simultaneously model the features of cycles in a graph and linear dependence in a vector space; see \cite{Oxley} for background.

Roughly speaking, an \textit{orderly matroid} is an oriented matroid whose circuits are cyclically ordered, and the ordering is \textit{compatible} with the orientation.
A motivating example is the \textit{cycle matroid} of a graph $\Gamma$: the ground set is the edge set of $\Gamma$, and the circuits are the edge sets of cycles in $\Gamma$. 
An orientation can be defined by choosing an orientation of each edge, and the circuits are cyclically ordered according to the order in which their elements appear along the corresponding cycles in $\Gamma$; see \S\ref{sec:cycle matroids}.
More generally, a matroid is \textit{graphic} if it is isomorphic to the cycle matroid of some graph.

\subsection{Orderly matroids}
We now extend the notion of orderability introduced by Crenshaw and Oxley~\cite{oxleycrenshaw} from matroids  to oriented matroids.

\subsubsection{Orderable matroids}
Let $X$ be a set with $n$ elements. A \textit{reversible cyclic ordering} of $X$ is a one-to-one assignment of the elements of $X$ to the vertices of an $n$-gon. Two elements are called \textit{adjacent} if the corresponding vertices of the $n$-gon are adjacent.
Let $M=(E,\mathcal{C})$ be a matroid. An \textit{ordering} of $M$ is a choice of a reversible cyclic ordering for each circuit. An ordering is called \textit{consistent} if, for any $C_1,C_2\in\mathcal{C}$ and $e,f\in C_1\cap C_2$, whenever $e$ and $f$ are adjacent in $C_1$, they are also adjacent in $C_2$. A matroid is called \textit{orderable} if it admits a consistent ordering. For example, graphic matroids are orderable; see~\cite[Proposition 1]{oxleycrenshaw}.

\subsubsection{Oriented matroids}
A \textit{signed set} $X$ consists of a set $\underline{X}$ together with a partition $(X^+,X^-)$ of $\underline{X}$.
We call $\underline{X}$ the \textit{support}, $X^+$ the set of \textit{positive} elements, and $X^-$ the set of \textit{negative} elements of $X$.
The \textit{opposite} of $X$, denoted by $-X$, is the signed set with the partition $(-X)^+=X^-$ and $(-X)^-=X^+$.
The \textit{empty signed set} is denoted by $\emptyset=(\emptyset,\emptyset)$. 
Let $X$ and $Y$ be two signed sets. We say that $X$ is a \textit{subset} of $Y$, denoted by $X\subseteq Y$, if and only if $X^+\subseteq Y^+$ and $X^-\subseteq Y^-$. We say that $X$ and $Y$ are equal, denoted by $X=Y$, if and only if $X\subseteq Y$ and $Y\subseteq X$. Let $E$ be a finite set. A \textit{signed subset} of $E$ is signed set $X$ such that $\underline{X}\subseteq E$.

An \textit{oriented matroid} is a pair $\mathcal{M}=(E,\mathcal C)$, where $E$ is a finite set, and $\mathcal{C}$ is a collection of signed subsets of $E$, satisfying the following axioms:
\begin{enumerate}
\item[(C0)] $\emptyset\notin\mathcal{C}$.
\item[(C1)] If $C\in\mathcal{C}$, then $-C\in\mathcal{C}$. 
\item[(C2)] If $C_1,C_2\in\mathcal{C}$ and $\underline{C_1}\subseteq\underline{C_2}$, then $C_1=C_2$ or $C_1=-C_2$.
\item[(C3)] If $C_1,C_2\in\mathcal{C}$, $C_1\neq-C_2$, and $e\in C_1^+\cap C_2^-$, then there exists $C_3\in\mathcal{C}$ such that $C_3^+\subseteq(C_1^+\cup C_2^+)\setminus\{e\}$ and $C_3^-\subseteq(C_1^-\cup C_2^-)\setminus\{e\}$.
\end{enumerate}
The set $E$ is called  the \textit{ground set}, and the elements of $\mathcal{C}$ are called the \textit{signed circuits}.
We refer the reader to~\cite{orientedmatroids,Laura2025} for more details.

\subsubsection{Orderly matroids}
Let $\mathcal M=(E,\mathcal{C})$ be an oriented matroid. 
An \textit{ordering} of $\mathcal{M}$ is a choice of a cyclic ordering on the support of each signed circuit $C\in\mathcal{C}$, such that the cyclic ordering assigned to $-C$ is the reverse of the cyclic ordering assigned to $C$.
We denote by $\omega$ the collection of the cyclic orderings of all signed circuits. Notice that these are genuine cyclic orderings rather than the reversible orderings considered in \cite{oxleycrenshaw}.

To keep track of both the orientation and the ordering, we use the following notation.
For $C\in \mathcal C$ with support $\underline{C}=\{e_1,\dots,e_n\}$, we write  $\omega(C)=\order{e_{i_1}^{\epsilon_{i_1}},\dots,e_{i_n}^{\epsilon_{i_n}}}$
to indicate that $\underline{C}$ is cyclically ordered as $(e_{i_1},\dots,e_{i_n})$
and that $e_{i_j}\in C^\pm$ when  $\epsilon_{i_j}=\pm 1$.
In this case,  we say that $e_{i_{j+1}}^{\epsilon_{i_{j+1}}}$  \textit{follows} $e_{i_j}^{\epsilon_{i_{j}}}$, with indices taken modulo~$n$. 
An ordering $\omega$ of $\mathcal{M}=(E,\mathcal{C})$ is \textit{compatible} if, for all $e_i,e_j\in E$ and all $C_1,C_2 \in \mathcal C$ with $e_i,e_j\in \underline{C_1}\cap\underline{C_2}$, 
whenever $e_j^{\epsilon_{j}}$ follows $e_i^{\epsilon_{i}}$ in  $\omega(C_1)$,
we have that 
$e_j^{\epsilon_{j}}$ follows $e_i^{\epsilon_{i}}$ in either $\omega(C_2)$ or $\omega(-C_2)$.
We refer the reader to Example~\ref{ex: OMG uniform} and Example~\ref{ex: OMG graphic} for the use of this notation.

\begin{definition}[Orderly matroid]\label{def: orderly matroid}
An \textit{orderly matroid} is a triple $\mathfrak M=(E,\mathcal{C},\omega)$, where $\mathcal M=(E,\mathcal{C})$ is an oriented matroid and $\omega$ is a compatible ordering of $\mathcal{M}$.
\end{definition}

Note that if $\mathfrak M$ is an orderly matroid, then the underlying (unoriented) matroid is orderable in the sense of \cite{oxleycrenshaw}.
Also, note that a matroid can be oriented and ordered without being orderly; see Example~\ref{ex: OMG uniform}.

The construction in \S\ref{sec:OMG} below can be carried out for matroids that are oriented and ordered, without assuming that the ordering is compatible. However, it seems more natural to develop the theory in the special case of orderly matroids. Compatibility will be particularly relevant in \S\ref{sec: orderly graph}.

\subsection{Orderly matroid groups}\label{sec:OMG}
We now associate a group to each orderly matroid. This construction is inspired by the Dicks--Leary presentation~\eqref{eq:infdicksleary} for BBGs.

\begin{definition}[Orderly matroid group]\label{def:OMG}
Let $\mathfrak M=(E,\mathcal{C},\omega)$ be an orderly matroid. 
The \textit{orderly matroid group (OMG)}, denoted by $\bbgm{\mathfrak M}$, is the group defined by the presentation 
\begin{equation}\label{eq:OMG}
\bbgm{\mathfrak M} =\langle E  \mid R_{\mathcal C,\omega}\rangle,
\end{equation}
where $R_{\mathcal C,\omega}$ consists of relators of the form 
$e_{i_1}^{\epsilon_{i_1}k}\cdots e_{i_n}^{\epsilon_{i_n}k}$
for each ordered signed circuit  $\omega(C)=\order{e_{i_1}^{\epsilon_{i_1}},\dots,e_{i_n}^{\epsilon_{i_n}}}$ and each $k\in\zz$.
\end{definition}

Some remarks are in order.

\begin{remark}
The presentation~\eqref{eq:OMG} is often redundant. 
For instance, if there is a circuit $C$ with $\omega(C)=\order e$, then $e=1$ in $\bbgm{\mathfrak M}$; if  there is a circuit $C$ with $\omega(C)=\order {e,f}$, then $f=e^{-1}$ in $\bbgm{\mathfrak M}$.
More generally, every generator that appears in a circuit can be written as a word in the other elements of that circuit. Therefore, any \textit{basis} $\mathcal{B}\subseteq E$ of $\mathfrak M$ provides a generating set of $\bbgm {\mathfrak{M}}$. 
\end{remark}

\begin{remark}
    The ordered signed circuits $\omega(C)$ and $\omega(-C)$ give rise to the same set of relators. Indeed, for each $k\in\zz$, the relator associated to $\omega(-C)$ is of the form $e_{i_n}^{-\epsilon_{i_n}k}\cdots e_{i_1}^{-\epsilon_{i_1}k}$, which is the inverse of the relator $e_{i_1}^{\epsilon_{i_1}k}\cdots e_{i_n}^{\epsilon_{i_n}k}$ given by $\omega(C)$. 
\end{remark}

\begin{remark}
    We will see in \S\ref{sec:orientation} that reorienting the matroid $\mathfrak M$ does not change the group $\bbgm{\mathfrak M}$ up to isomorphism.
    On the other hand, the ordering plays a more essential role in Definition~\ref{def:OMG}. 
    Sometimes the ordering can be ignored (see Corollary~\ref{cor: s.c. OMG}), but sometimes it cannot (see Remark~\ref{jialin}).
\end{remark}

We now give examples of OMGs of \textit{orderly uniform matroids}. In~\S\ref{sec:cycle matroids}, we will focus on the OMGs arising from orderly matroids associated with graphs. 

\begin{example}\label{ex: OMG uniform} 
A \textit{uniform matroid} $U_{r,n}$ is a matroid on an $n$-element set $E=\{e_1,\dots,e_n\}$ whose circuits are precisely the $(r+1)$-element subsets of $E$.
It was shown in \cite[Corollary 17]{oxleycrenshaw} that the orderable uniform matroids are exactly $U_{0,n}$, $U_{1,n}$, $U_{2,n}$, $U_{n-1,n}$, and $U_{n,n}$.
Among them, only $U_{2,n}$ is not orderly and not graphic for $n\geq 4$. The OMGs of the others are as follows.
    \begin{enumerate}
        \item \label{item:U0n} $\bbgm{U_{0,n}}\cong 1$. Indeed, every single-element subset of $E$ is a circuit, so the list of relators includes $e_i$ for $i=1,\dots, n$.

        \item $\bbgm{U_{1,n}}\cong\mathbb{Z}$.
        The ordered signed circuits are $\order{e_i^{+1},e_j^{-1}}$ and $\order{e_j^{+1},e_i^{-1}}$ for all $i<j$. The corresponding relators include $e_ie_j^{-1}$. Thus, all the generators are identified, and all the relators involving higher powers are redundant.
        
        \item $\bbgm{U_{n-1,n}}\cong \bbg{C_n}$, where $C_n$ is the cycle of length $n$. This group is $\zz^2$ for $n=3$, and it is not finitely presented for $n\geq4$; see \cite{BB1997}.
        There is only one circuit $C$ with support $\underline{C}=E$. We can orient the matroid and order the signed circuit arbitrarily, say $C=\order{e_1^{+1},\dots,e_n^{+1}}$. The corresponding relators are of the form $e_1^k\cdots e_n^k$ for each $k\in\zz$, which is the set of relators in the Dicks--Leary presentation~\eqref{eq:infdicksleary} of $\bbg{C_{n}}$.
        
        \item $\bbgm{U_{n,n}}\cong \ff_n$. Indeed, there are no signed circuits, so there are no relators.         
    \end{enumerate}
Note that the orderly uniform matroids are exactly the orderly cycle matroids of the graphs in Example~\ref{prop:uniform_raag} below.
\end{example}

\subsubsection{Reorientation}\label{sec:orientation}
Let $\mathcal{M}=(E,\mathcal{C})$ be an oriented matroid, and let $A\subseteq E$. A \textit{reorientation} of $\mathcal{M}$ at $A$, denoted by $\mathcal{M}_{\overline{A}}=(E,\mathcal{C}_{\overline{A}})$, is the oriented matroid obtained from $\mathcal{M}$ by reversing the signs of all elements of $A$ in each signed circuit of $\mathcal{C}$.
For an orderly matroid $\mathfrak{M}=(E,\mathcal{C},\omega)$ and $A\subseteq E$, we define the \textit{reorientation} at $A$ by reorienting the underlying oriented matroid $\mathcal{M}=(E,\mathcal{C})$ at $A$ and keeping the same ordering $\omega$.
This gives another orderly matroid $\mathfrak{M}_{\overline{A}}=(E,\mathcal{C}_{\overline{A}},\omega)$, in the sense that $\omega$ is compatible with the reorientation.
More precisely, if $\omega(C)=\order{e_{i_1}^{\epsilon_{i_1}},\dots,e_{i_n}^{\epsilon_{i_n}}}$
is an ordered signed circuit of $\mathfrak{M}$, then the corresponding ordered signed circuit of $\mathfrak{M}_{\overline A}$ is
$$\omega(C_{\overline A})=\order{e_{i_1}^{a_{i_1}\epsilon_{i_1}},\dots,e_{i_n}^{a_{i_n}\epsilon_{i_n}}},$$
where $a_{i_j}=-1$ if $e_{i_j}\in A$ and $a_{i_j}=+1$ if $e_{i_j}\notin A$.

The next statement shows that the OMG of an orderly matroid is an invariant of the reorientation class.
Note that graphic matroids have a single reorientation class; see \S\ref{subsec:graphic}.

\begin{lemma}\label{lem:reorientation}
    Let $\mathfrak{M}=(E,\mathcal{C},\omega)$ be an orderly matroid. For any $A\subseteq E$,  $\bbgm{\mathfrak{M}}\cong\bbgm{\mathfrak{M}_{\overline{A}}}$.
\end{lemma}

\begin{proof}
    Let $E=\{e_1,\dots,e_n\}$, and let $F(E)$ be the free group generated by $E$. Define an automorphism $\phi_A\colon F(E)\to F(E)$ by 
    $$
    \phi_A(e_i)
    =
    \begin{cases}
    e_i^{-1}, & e_i\in A,
    \\ 
    e_i, & e_i\notin A.
    \end{cases}
    $$
    The map $\phi_A$ sends  each relator in the presentation \eqref{eq:OMG} of $\bbgm{\mathfrak{M}}$ associated to $\omega (C)$ to the corresponding relator in the presentation  of $\bbgm{\mathfrak{M}_{\overline{A}}}$ associated to $\omega (C_{\overline A})$.
    Thus, the map $\phi_A$
    descends to a homomorphism $\bbgm{\mathfrak{M}} \to \bbgm{\mathfrak{M}_{\overline{A}}}$. Since $\phi_A=\phi_A^{-1}$, we can construct a homomorphism $\bbgm{\mathfrak{M}_{\overline{A}}}\to\bbgm{\mathfrak{M}}$ in a similar way. Hence, $\bbgm{\mathfrak{M}}\cong\bbgm{\mathfrak{M}_{\overline{A}}}$.
\end{proof}

\subsubsection{Cycle matroids}\label{sec:cycle matroids}\label{subsec:graphic}
We now give the main motivating example for Definition~\ref{def: orderly matroid}.
The \textit{cycle matroid} $M_\Gamma$ of a graph $\Gamma$ is the matroid with ground set $\ee \Gamma$, and the set of circuits $\mathcal{C}(\Gamma)$ consists of the edge sets of cycles in $\Gamma$.
Here, we do not assume that $\Gamma$ is simplicial.

A \textit{digraph} is a graph with a chosen orientation of each edge.
Given this orientation, the edge set of each cycle can be partitioned into positive and negative parts, according to whether the edges are traversed with or against their chosen orientations.
Thus, every digraph $\Gamma$ defines an oriented matroid $\mathcal M_\Gamma$, called the \textit{oriented cycle matroid} of $\Gamma$. 
Moreover, every oriented cycle comes with a natural cyclic ordering $\omega_\Gamma$, called the \textit{graphic ordering}; this is a compatible ordering.
Hence, every digraph $\Gamma$ defines an orderly matroid $\mathfrak M_\Gamma=(\ee\Gamma, \mathcal{C}(\Gamma),{\omega_\Gamma})$, called the \textit{orderly cycle matroid} of $\Gamma$.

Note that if $\Gamma$ is a digraph, not necessarily simplicial, and $\Gamma'$ is the simplicial digraph obtained from $\Gamma$ by removing all loops and identifying parallel edges, then $\bbgm{\mathfrak M_\Gamma} \cong \bbgm{\mathfrak M_{\Gamma'}}$. This can be shown directly from the relators in the presentation \eqref{eq:OMG}. Also, compare Remark~\ref{remark:non simplicial RAAG BBG}. 

\begin{example}\label{ex:OOM}\label{ex: OMG graphic}
Consider the digraph $\Gamma$ shown in Figure~\ref{fig: two triangles}. 
The orderly cycle matroid $\mathfrak M_\Gamma$ of $\Gamma$ has ground set $\ee\Gamma=\{a,b,c,d,e\}$. 
The signed circuits are cyclically ordered by the graphic ordering $\omega_\Gamma$ as follows:
$$
\begin{aligned}
\omega_\Gamma(C_1)&=\order{a^{+1},d^{+1},b^{+1}},
&\;
\omega_\Gamma(-C_1)&=\order{b^{-1},d^{-1},a^{-1}},\\
\omega_\Gamma(C_2)&=\order{b^{+1},c^{-1},e^{-1}},
&\;
\omega_\Gamma(-C_2)&=\order{e^{+1},c^{+1},b^{-1}},\\
\omega_\Gamma(C_3)&=\order{a^{+1},d^{+1},e^{+1},c^{+1}},
&\;
\omega_\Gamma(-C_3)&=\order{c^{-1},e^{-1},d^{-1},a^{-1}}.
\end{aligned}
$$

Note that $\omega_\Gamma$ is the only compatible ordering on this oriented matroid: the ordering of the square $C_3$ is determined by the ordering of the two triangles $C_1$ and $C_2$.
The group $\bbgm{\mathfrak{M}_\Gamma}$ is generated by $E=\{a,b,c,d,e\}$, and the relators associated to $\omega_\Gamma(C_i)$ are $a^kd^kb^k$, $b^kc^{-k}e^{-k}$ and $a^kd^ke^kc^k$ for each $k\in\zz$.
Using $k=\pm1$ in the first two sets of relators gives $d=a^{-1}b^{-1}=b^{-1}a^{-1}$ and $e=bc^{-1}=c^{-1}b$. Thus, both $\{a,b,d\}$ and $\{b,c,e\}$ generate a $\zz^2$ subgroup. In particular, the generator $b$ is central.
Therefore, the presentation reduces to $\langle a,b,c\mid [a,b], [b,c]\rangle$. Hence,  $\bbgm{\mathfrak{M}_\Gamma}\cong \mathbb{Z}\times \ff_2$.

For completeness, here is a consistent but not compatible ordering $\omega'$ of $\mathfrak M_\Gamma$: 
$$ 
\omega'(\pm C_i)= \omega_\Gamma (\pm C_i) \textrm{ for } i=1,2,$$
$$\omega'(C_3)=\order{a^{+1},d^{+1},c^{+1},e^{+1}}, \quad \omega'(-C_3)=\order{e^{-1},c^{-1},d^{-1},a^{-1}}.
$$
The reader can verify that the presentation with respect to $\omega'$ also gives $\mathbb{Z}\times \ff_2$.
\end{example}

\begin{figure}[ht!]
\centering\usetikzlibrary{decorations.markings, arrows.meta}

\tikzset{
    mid arrow/.style={
        postaction={
            decorate,
            decoration={
                markings,
                mark=at position 0.55 with {\arrow[scale=1]{Triangle}}
            }
        }
    }
}

\begin{tikzpicture}[scale=0.7]
\draw[thick, mid arrow] (3,3) -- (0,1.5);
\draw[thick, mid arrow] (0,1.5) -- (3,0);
\draw[thick, mid arrow] (3,0) -- (6,1.5);
\draw[thick, mid arrow] (6,1.5) -- (3,3);
\draw[thick, mid arrow] (3,0) -- (3,3);
\draw [fill] (0,1.5) circle [radius=0.15];
\draw [fill] (3,0) circle [radius=0.15];
\draw [fill] (3,3) circle [radius=0.15];
\draw [fill] (6,1.5) circle [radius=0.15];
\node [above left] at (1.5,2.25) {$a$};
\node [below left] at (1.5,0.75) {$d$};
\node [below right] at (4.5,0.75) {$e$};
\node [above right] at (4.5,2.25) {$c$};
\node [left] at (3,1.6) {$b$};
\end{tikzpicture}
    \caption{A digraph.}
    \label{fig: two triangles}
\end{figure}
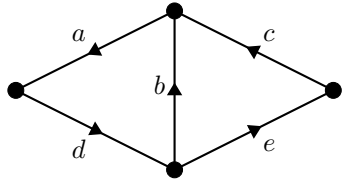

We are ready to prove Theorem~\ref{mainthm matroid intro}, that is, the construction developed so far recovers the BBG of a graph.
Recall from Remark~\ref{remark:non simplicial RAAG BBG} that $\bbg \Gamma$ is defined even when $\Gamma$ is a non-simplicial graph.

\begin{proof}[Proof of Theorem~\ref{mainthm matroid intro}]
We know from Lemma~\ref{lem:reorientation} that 
the assignment $\mathfrak M\mapsto \bbgm{\mathfrak M}$ is insensitive to reorientation. Thus, the bottom-right map in the diagram is well-defined.
Moreover, every oriented cycle matroid has a unique reorientation class (see \cite[Proposition 6.2]{BLV78}), so the bottom-left map in the diagram is also well-defined.

Now, let $\Gamma$ be a connected graph and fix an arbitrary orientation of the edges of $\Gamma$. 
Let $\mathfrak M_\Gamma=(\ee\Gamma,\mathcal C(\Gamma), \omega_\Gamma)$ be the associated orderly cycle matroid, and let $\bbgm{\mathfrak M_\Gamma}$ be its OMG. 
As noted above, the group $\bbgm{\mathfrak M_\Gamma}$ does not depend on the chosen orientation.
As observed in \cite{DicksLeary99}, the Dicks--Leary presentation \eqref{eq:infdicksleary} for the group $\bbg \Gamma$ can be simplified as follows: use the relators in $I_\Gamma$ to halve the number of generators, and then introduce signs into the exponents of the relators in $R_\Gamma$. This yields the presentation \eqref{eq:OMG} of $\bbgm {\mathfrak M_\Gamma}$.
\end{proof}

This means that the BBG of a graph depends only on the orderly cycle matroid of the defining graph, that is, the cycle matroid equipped with the graphic ordering.
A natural question is how much of the orderly structure is actually needed. In the next section, we clarify the role of the ordering and in particular, we show that for a graphic matroid, it corresponds to the choice of a defining graph.
Remark~\ref{jialin} shows that in general this choice is delicate. However, the choice of ordering turns out to be irrelevant for finitely presented BBGs (the ones defined by simply connected flag complexes), as shown in the next result.

\begin{corollary}\label{cor: s.c. OMG}
Let $\Gamma_1$ and $\Gamma_2$ be  simplicial graphs with simply connected flag complexes and isomorphic cycle matroids.
Then, $ \bbgm{\mathfrak M_{\Gamma_1}} \cong \bbgm{\mathfrak M_{\Gamma_2}}$.
\end{corollary}
\begin{proof}
If a graph has a cut vertex, then its cycle matroid is the direct sum of the cycle matroids of its biconnected components, and the associated OMG is a free product of the OMGs correspond to the biconnected components. Thus, we may assume that both $\Gamma_1$ and $\Gamma_2$ are biconnected.
Since $M_{\Gamma_1}\cong M_{\Gamma_2}$, Whitney's 2-isomorphism theorem implies that $\Gamma_1$ and $\Gamma_2$ are 2-isomorphic; see \cite[Theorem 5.3.1]{Oxley}. Since both $\flag{\Gamma_1}$ and $\flag{\Gamma_2}$ are simply connected, Proposition~\ref{prop:2isobbg} implies $\bbg{\Gamma_1}\cong\bbg{\Gamma_2}$.
Hence, we have $\bbgm{\mathfrak M_{\Gamma_1}}\cong \bbgm{\mathfrak M_{\Gamma_2}}$ by Theorem~\ref{mainthm matroid intro}.
\end{proof}

\subsection{Orderly graphs}\label{sec: orderly graph}
In this section, we show that given an orderly matroid, one can canonically construct a BBG that is a quotient of the OMG of the given orderly matroid. 
We do this by first defining a graph associated to the orderly matroid; this is where the notion of compatibility becomes relevant.
It turns out that the OMG of an orderly graphic matroid is isomorphic to the BBG obtained from this construction.

\begin{definition}[Orderly graph]\label{def:orderly graph}
Let  $\mathfrak M =(E,\mathcal C, \omega)$ be an orderly matroid.
Let $V=\{h_e,t_e \mid e\in E\}$, where $h_e$ and $t_e$ are formal symbols.
Let $\sim$ be the equivalence relation on $V$ generated by the following:
for each $C\in \mathcal C$ and  $e,f\in\underline{C}$,
\begin{equation}\label{eq:orderly graph relators}
\begin{cases}
h_e\sim t_f, & \text{if } f^{+1} \text{ follows } e^{+1} \text{ in } \omega(C),\\
t_e\sim h_f, & \text{if } f^{-1} \text{ follows }  e^{-1} \text{ in } \omega(C),\\
h_e \sim h_f, & \text{if } f^{-1} \text{ follows } e^{+1} \text{ in } \omega(C),\\
t_e \sim t_f, & \text{if } f^{+1} \text{ follows } e^{-1} \text{ in } \omega(C).
\end{cases}
\end{equation}
(If $C$ consists of a single element $e\in E$, then it is understood that $e^{+1}$ follows $e^{+1}$ in $\omega(C)$. Therefore, we have  $h_e\sim t_e$, that is, the element $e$ gives a loop.)
The \textit{orderly graph} $\shadow{\mathfrak M}$ of $\mathfrak M$ is the graph with vertex set $V/\sim$, and there is an edge between $[h_e]$ and $[t_e]$ for each $e\in E$. Notice that $\ee{\shadow{\mathfrak M}}=E$.
\end{definition}

\begin{remark}\label{rmk:TH orientation}
    We think of $h_e$ and $t_e$ as the ``head'' and ``tail'' of $e$. More formally, we orient every edge of $\shadow{ \mathfrak M}$ from tail to head.
    Then, each oriented edge of $\shadow{ \mathfrak M}$ with distinct endpoints corresponds to a unique generator (or the inverse of a generator) in $\bbgm{\mathfrak M}$. On the other hand, loops correspond to the trivial element in $\bbgm{\mathfrak M}$.
    Therefore, every oriented edge-path  in $\shadow{ \mathfrak M}$ corresponds to a unique  element in $\bbgm{\mathfrak M}$.
\end{remark}

\begin{remark}
The identifications in Definition~\ref{def:orderly graph} depend only on the elements $e$ and $f$ but not on the choice of the ordered signed circuit $C$.
For example, in the first case, the element $f^{+1}$ follows $e^{+1}$ in $\omega(C)$, then $e^{-1}$ follows $f^{-1}$ in $\omega(-C)$ by the definition of ordering. Since the ordering is compatible, if $e,f\in \underline{D}$ for another ordered signed circuit $D\in \mathcal C$, then $f^{+1}$ follows $e^{+1}$ in either $\omega(D)$ or $\omega(-D)$. 
\end{remark}

\begin{example}\label{prop:uniform_raag}
The orderly graphs of orderly uniform matroids can be constructed directly from the definition. (See  Example~\ref{ex: OMG uniform} for the associated OMGs.)
\begin{enumerate}
    \item \label{item:AU0n} 
    $\shadow{U_{0,n}}$ is a disjoint union of $n$ loops. 
        
    \item $\shadow{U_{1,n}}$ consists of two vertices joined by $n$ parallel edges. 
        
    \item $\shadow{U_{n-1,n}}$ is the cycle of length $n$ for $n\geq3$.
        
    \item $\shadow{U_{n,n}}$ consists of $n$ disjoint edges since there are no circuits.
\end{enumerate}
\end{example}

Definition~\ref{def:orderly graph} is inspired by cycle matroids.
We have the following lemma. 

\begin{lemma}\label{lem:orderly graph of a graph}
    Let $\Gamma$ be a  biconnected graph, and let $\mathfrak M_\Gamma$ be the associated orderly cycle matroid. Then, $\Gamma \cong \shadow{\mathfrak M_\Gamma}$.
\end{lemma}
\begin{proof}
    Fix an arbitrary orientation of the edges of $\Gamma$ and define a map $\alpha\colon\vv \Gamma\to \vv{\shadow{\mathfrak M_\Gamma}}$ as follows.
    Let $v\in \vv \Gamma$, and let $e$ be an edge incident to $v$. If $e$ is outgoing from $v$, then set $\alpha(v)=[t_e]$; if $e$ is incoming to $v$, then set $\alpha(v)=[h_e]$.
    Since $\Gamma$ is biconnected, any two edges meeting at $v$ belong to a common cycle. Therefore, the equivalence relation $\sim$ identifies all the heads and tails based at $v$. Thus, the map $\alpha$ is well-defined and is surjective by construction.
    Since the identification defining $\sim$ occurs only when edges of $\Gamma$ meet at a common vertex, each equivalence class consists of formal heads and tails based at a single vertex of $\Gamma$. It follows that $\alpha$ is injective. 
    Moreover, by  construction, the edge set of $\shadow{\mathfrak M_\Gamma}$ coincides with the ground set of $\mathfrak M_\Gamma$, which is the edge set of $\Gamma$. Therefore, we can extend $\alpha$ to a graph isomorphism $\Gamma\to \shadow{\mathfrak M_\Gamma}$.
\end{proof}

\subsubsection{Mapping to a BBG}
We now consider the RAAGs defined by orderly graphs.
Recall from Remark~\ref{remark:non simplicial RAAG BBG} that RAAGs are defined even if their defining graphs are not simplicial.
We say that an orderly matroid is \textit{connected} if any two elements of the ground set are contained in the support of a common ordered signed circuit.

\begin{lemma}\label{lem:connected}
If $\mathfrak M= (E,\mathcal C, \omega)$ is a connected orderly matroid, then $\shadow{\mathfrak M}$ is connected.
\end{lemma}
\begin{proof}
Let $v=[t_e]$ and $w=[h_f]$ be two distinct vertices of $\shadow{\mathfrak M}$ for some $e,f\in E$. If $e=f$, then $v$ and $w$ are adjacent in $\shadow{\mathfrak M}$. If $e\neq f$, then there is an ordered signed circuit whose support contains both $e$ and $f$ since $\mathfrak M$ is connected. 
This provides (at least) one  edge-path in $\shadow{\mathfrak M}$ from $v$ to $w$. Thus, the graph $\shadow{\mathfrak M}$ is connected.
\end{proof}

\begin{proposition}\label{prop:AOMG=RAAG}
Let $\mathfrak M = (E,\mathcal C, \omega)$ be an orderly matroid. Consider the map  $\phi\colon E\to \raag{\shadow{\mathfrak M}}$ defined by $\phi(e)=[t_e][h_e]^{-1}$. Then the following statements hold.
\begin{enumerate}
    
    \item \label{item: AOMG map} The map  $\phi\colon E\to \raag{\shadow{\mathfrak M}}$ extends uniquely to a  homomorphism $\phi\colon \bbgm{\mathfrak M}\to \raag{\shadow{\mathfrak M}}$.

    \item \label{item: AOMG to BBG} $\mathrm{im}(\phi) \subseteq \bbg{\shadow{\mathfrak M}}$.

    \item \label{item: AOMG connected} If $\mathfrak M$ is connected, then $\mathrm{im}(\phi) = \bbg{\shadow{\mathfrak M}}$.
    
    \item \label{item: AOMG disconnected}  If $\mathfrak M$  has connected components $\mathfrak M_1,\dots,\mathfrak M_n$, then $$\mathrm{im}(\phi) \cong \Asterisk_{i=1}^n \bbg{\shadow{\mathfrak{M}_i}}\cong \bbg{\bigvee_{i=1}^n \shadow{\mathfrak M_i}}.$$ 
\end{enumerate} 
\end{proposition}

\begin{proof}
We need to verify that $\phi$ sends the relators in the presentation of $\bbgm{\mathfrak M}$ from Definition~\ref{def:OMG} to the identity.
For each  ordered signed circuit $\omega(C)=\order{e_1^{\epsilon_1},e_2^{\epsilon_2},\dots,e_n^{\epsilon_n}}$  and  $k\in\zz$, we have the relator $
e_1^{\epsilon_1k}e_2^{\epsilon_2k}\cdots e_n^{\epsilon_nk}$.
    Since $[t_e]$ and $[h_e]$ are either equal or adjacent in $\shadow{\mathfrak M}$, they commute in $\raag{\shadow{\mathfrak M}}$, so for each $i=1,\dots,n$, we have
    $$
    \phi(e_i^{\epsilon_i k})=
    \begin{cases}
    [t_{e_i}]^k[h_{e_i}]^{-k}, & \epsilon_i=1,\\[4pt]
    [h_{e_i}]^k[t_{e_i}]^{-k}, & \epsilon_i=-1.
    \end{cases}
    $$
    In either case, the element  $e_{i+1}^{\epsilon_{i+1}}$ follows $e_i^{\epsilon_i}$ (indices modulo $n$), so the relators in~\eqref{eq:orderly graph relators} in Definition~\ref{def:orderly graph} identify the second half of $\phi(e_i^{\epsilon_ik})$ with the inverse of the first half of $\phi(e_{i+1}^{\epsilon_{i+1}k})$. Therefore, these two adjacent factors cancel, and we obtain $\phi(e_1^{\epsilon_1k}e_2^{\epsilon_2k}\cdots e_n^{\epsilon_nk})=1$.  
    This proves \eqref{item: AOMG map}.

    We now consider the image of $\phi$. 
    Let $\chi\colon \raag{\shadow{\mathfrak M}}\to \zz$ be the homomorphism sending every generator to $1$. 
    Then by definition $\bbg{\shadow{\mathfrak M}} = \ker\chi$. For each $e\in E$, we have $\chi(\phi(e))=\chi([t_e][h_e]^{-1})=1-1=0$. Therefore, $\mathrm{im}(\phi)\subseteq\bbg{\shadow{\mathfrak M}}$.
    This proves \eqref{item: AOMG to BBG}.
    
    To prove \eqref{item: AOMG connected}, we follow the proof of \cite[Theorem 1]{DicksLeary99}.
    Let $h\in \bbg{\shadow{\mathfrak M}}=\ker\chi$.
    Then $h=v_1^{n_1}\cdots v_m^{n_m}$ for some $v_1,\dots,v_m\in\vv{\shadow{\mathfrak M}}$ and $n_1,\dots,n_m\in\zz$ with $n_1+\dots + n_m=0$.
    Since $\shadow{\mathfrak M}$ is connected  by Lemma~\ref{lem:connected}, we can choose an edge-path  from $v_m$ to $v_{m-1}$ in $\shadow{\mathfrak M}$, and let $f_1,\dots,f_r \in E \cup E^{-1}\subseteq \bbgm{\mathfrak M}$ be the elements that correspond to its edges (in the sense of Remark~\ref{rmk:TH orientation}).
    Let $x=f_1^{-n_m}\cdots f_r^{-n_m}  \in \bbgm{\mathfrak M}$. Then $\phi(x)=v_m^{-n_m}v_{m-1}^{n_m}$, and it follows by induction on $m$ that $h\in \textrm{im}(\phi)$.
    This proves \eqref{item: AOMG connected}.

    Finally, assume that $\mathfrak M$ is disconnected, and let $\mathfrak M_1,\dots, \mathfrak M_n$ be its connected components (see \cite[\S 4.2]{Oxley} for the definitions).
    Since no ordered signed circuit contains elements from two different connected components, it follows from the presentation \eqref{eq:OMG} that $\bbgm{\mathfrak M} \cong \Asterisk_{i=1}^n \bbgm{\mathfrak M_i}$ and $\shadow{\mathfrak M} = \bigsqcup_{i=1}^n \shadow{\mathfrak M_i}$.
    By \eqref{item: AOMG connected}, there is a surjection $\phi_i\colon\bbgm{\mathfrak M_i} \to \bbg{\Gamma_{\mathfrak M_i}}$ for each $i=1,\dots,n$, and the free product of $\phi_1,\dots,\phi_n$ gives a surjection $\bbgm{\mathfrak M} \to \Asterisk_{i=1}^n \bbg{\Gamma_{\mathfrak M_i}}$.
    Note that  
    $\Asterisk_{i=1}^n \bbg{\Gamma_{\mathfrak M_i}}$ is the BBG associated to the graph $\bigvee_{i=1}^n \shadow{\mathfrak M_i}$, which is obtained from $\shadow{\mathfrak M} = \bigsqcup_{i=1}^n \shadow{\mathfrak M_i}$ by identifying one  vertex from each $\shadow{\mathfrak M_i}$. This proves \eqref{item: AOMG disconnected}.
\end{proof}

In particular, we obtain the following criteria for estimating the size of the OMG. In general, the cardinality of the ground set of $\mathfrak M$ does not give a lower bound on the rank of $\bbgm{\mathfrak M}$. For example, $\bbgm{U_{0,n}}=1$ for all $n\in \nn$; see Example~\ref{ex: OMG uniform}.

\begin{corollary}
Let $\mathfrak M$ be an orderly matroid. The following statements hold.
\begin{enumerate}
    \item If $\shadow{\mathfrak M}$ retracts to a subclique on $k+1$ vertices, then 
    $\bbgm{\mathfrak M}$ surjects $\zz^k$.

    \item If $\shadow{\mathfrak M}$ retracts to a subtree with $k+1$ vertices, then $\bbgm{\mathfrak M}$ surjects $\ff_k$.
    \end{enumerate}
\end{corollary}

\begin{remark}[Compatibility]
One could have developed the theory for matroids that are orderable and oriented but not necessarily orderly.
Indeed, if $\omega$ is not a compatible ordering of such a matroid $\mathfrak M=(E,\mathcal{C},\omega)$, then there are $e,f \in E$ such that the identifications in \eqref{eq:orderly graph relators} associated to two different ordered signed circuits force one of $e$, $f$, $ef$, or $ef^{-1}$ to lie in $\ker \phi$.
Consequently, the map $\phi\colon\bbgm{\mathfrak M}\to \raag{\shadow{\mathfrak M}}$ can fail to be injective. 
This is a good reason to restrict our attention to orderly matroids.
In the next section, we prove injectivity of $\phi$ for graphic matroids.
\end{remark}

\subsubsection{Graphic matroids}
If $\Gamma$ is a graph and $\mathfrak M_\Gamma$ is its orderly cycle matroid, then we know that $\bbgm{\mathfrak M_\Gamma}\cong \bbg \Gamma$ by Theorem~\ref{mainthm matroid intro}.
In this section, we prove Theorem~\ref{thm:faithful} that if $\mathfrak M$ is graphic, then   $\bbgm{\mathfrak M}$  is  isomorphic to a BBG (the one associated with the wedge sum of the  components of $\shadow{\mathfrak M}$), and that
the set of defining graphs for a graphic matroid is in bijection with a certain set of orderings of the  matroid; see Proposition~\ref{prop:defining graphs and orderings}.

Let $\mathfrak M=(E,\mathcal C,\omega)$ be an orderly matroid.
Recall that there is a canonical orientation on $\ee{\shadow{\mathfrak M}}=E$; see Remark~\ref{rmk:TH orientation}. 
In particular, every ordered signed circuit of $\mathfrak M$ induces an oriented cycle in $\shadow{\mathfrak M}$, but in general the converse may not be true.
Indeed, suppose that $(e_1,\dots,e_n)$ is a cycle in $\shadow{ \mathfrak M}$. Then, for each $i=1,\dots,n$, since $e_i$ and $e_{i+1}$ are consecutive in $\shadow{ \mathfrak M}$, there must be a circuit $C_i$ of $\mathfrak M$ in which both $e_i$ and $e_{i+1}$ appear consecutively. However, there is a priori no reason why there should be a circuit $C$ of $\mathfrak M$ in which all $e_1,\dots,e_n$ appear. 
We think of this as a local-to-global problem, and therefore we give the following definition.

\begin{definition}\label{def:globally compatible}
    Let $\mathfrak M =(E,\mathcal C, \omega)$ be an orderly matroid. We say that $\omega$ is \textit{globally compatible} if
    every  oriented cycle in $\shadow{\mathfrak M}$ is induced by an ordered signed circuit of $\mathfrak M$.
\end{definition}

This definition isolates the key property of a graphic matroid that is relevant in  the proofs of \cite[Theorem 1]{DicksLeary99} and Theorem~\ref{thm:faithful} below.
The next result says that a globally compatible ordering is a graphic ordering. However, it is not clear whether every compatible ordering is necessarily a graphic ordering, even when the underlying matroid is graphic.

\begin{lemma}\label{lem:faithful}
Let $\mathfrak M=(E,\mathcal{C},\omega)$ be an orderly matroid.
Then $\mathfrak M$ is graphic if and only if $\omega$ is globally compatible.
In this case, we have $\mathfrak M \cong \mathfrak M_{\shadow{\mathfrak M}}$ and $\omega=\omega_{\shadow{\mathfrak M}}$.
\end{lemma}
\begin{proof}
Suppose that $\mathfrak M$ is graphic. Fix a defining graph $\Gamma$ such that $\mathfrak M\cong \mathfrak M_\Gamma$ as an orderly matroid. In particular, this means that $\omega=\omega_\Gamma$. 
By definition, the ordered signed circuits of $\mathfrak M$ correspond exactly to the cycles of $\Gamma$. Moreover, we have $\Gamma\cong \shadow{\mathfrak M}$ by Lemma~\ref{lem:orderly graph of a graph}.

Conversely, suppose that $\omega$ is globally compatible.
We claim that $\mathfrak M$ is isomorphic to the cycle matroid of the orderly graph $\shadow{\mathfrak M}$.
Define a map $\xi\colon E\to\ee{\shadow{\mathfrak M}}$ by $\xi(e)=([t_e],[h_e])$; this is a bijection by construction. It suffices to show that $\xi$ maps ordered signed circuits of $\mathfrak M$ to ordered signed circuits of $\mathfrak M_{\shadow{\mathfrak M}}$.

Let $\omega(C)$ be an ordered signed circuit of $\mathfrak{M}$. Then $\xi$ sends $\omega(C)$ to a closed oriented edge-path $\widetilde{C}$ in $\shadow{\mathfrak M}$. If $\widetilde{C}$ were not an oriented cycle, then it would contain a proper oriented cycle $\widetilde{D}$. Since $\omega$ is globally compatible, the oriented cycle $\widetilde{D}$ is induced by an ordered signed circuit $\omega(D)$ of $\mathfrak{M}$.
Since $\xi$ is a bijection, we have $\underline{D}\subset\underline{C}$, which implies that $D=C$ or $D=-C$ by the definition of oriented matroid. In either case, we have $\underline{D}=\underline{C}$,and therefore $\widetilde{D}=\widetilde{C}$, which is a contradiction. Thus, the closed oriented edge-path $\widetilde{C}$ is an oriented cycle of $\shadow{\mathfrak M}$.

On the other hand, let $\widetilde{C}$ be an oriented cycle of $\shadow{\mathfrak M}$. Since $\omega$ is globally compatible, the oriented cycle $\widetilde{C}$ is induced by an ordered signed circuit $\omega(C)$ of $\mathfrak{M}$. Thus, the map $\xi$ sends $\widetilde{C}$ to $\omega(C)$. This proves $\mathfrak M \cong \mathfrak M_{\shadow{\mathfrak M}}$. Thus, the orderly matroid $\mathfrak M$ is graphic, and $\omega=\omega_{\shadow{\mathfrak M}}$. 
\end{proof}

We are ready to prove Theorem~\ref{thm:faithful}.

\begin{proof}[Proof of Theorem~\ref{thm:faithful}]
We first consider the case where $\mathfrak M$ is connected.
    Recall from Proposition~\ref{prop:AOMG=RAAG} that there is an epimorphism $\phi\colon \bbgm{\mathfrak M}\to \bbg{\shadow{\mathfrak M}}$ given by $\phi(e)=[t_e][h_e]^{-1}$ for $e\in E$.
    It remains to prove that $\phi$ is injective.

    Since $\mathfrak M$ is connected, the graph $\shadow{\mathfrak M}$ is connected by Lemma~\ref{lem:connected}.
    Let $v,w\in\vv{\shadow{\mathfrak M}}$. By Remark~\ref{rmk:TH orientation}, any edge-path from $v$ to $w$ gives a well-defined element of $\bbgm{\mathfrak M}$.    
     Moreover, if there are two oriented edge-paths from $v$ to $w$, then the concatenation of one with the reverse of the other gives an oriented closed edge-path $Q$ in $\shadow{\mathfrak M}$ that decomposes into oriented cycles. 
     Since $\mathfrak M$ is graphic, we know that $\omega$ is globally compatible by Lemma~\ref{lem:faithful}. Thus,  each of these oriented cycles is induced by an ordered signed circuit of $\mathfrak M$ and therefore corresponds to a relator in $\bbgm{\mathfrak M}$. It follows that the element of $\bbgm{\mathfrak M}$ representing $Q$ is trivial.
    Hence, the two edge-paths from $v$ to $w$ represent the same element of $\bbgm{\mathfrak M}$. We denote this element by $p(v,w)$.

    Now, fix $v\in \vv{\shadow{\mathfrak M}}$, and consider the map $\psi_v\colon E\to \bbgm{\mathfrak M}$ defined by 
    $$\psi_v(e)=p(v,[t_e])ep([t_e],v).$$
The computation in the proof of \cite[Theorem 1]{DicksLeary99} shows that  $\psi_v$ extends to an automorphism $\psi_v\colon \bbgm{\mathfrak M}\to \bbgm{\mathfrak M}$.
    Let $\zz =\langle s\rangle$, and consider $\bbgm{\mathfrak M} \rtimes_{\psi_v} \zz$. Then the map $\phi\colon \bbgm{\mathfrak M}\to \bbg{\shadow{\mathfrak M}}$ extends to an epimorphism 
    $$\widetilde \phi:\bbgm{\mathfrak M} \rtimes_{\psi_v} \zz \to \raag{\shadow{\mathfrak M}}$$
    by setting $\widetilde \phi(s)=v$.
Once again, as in the proof of \cite[Theorem 1]{DicksLeary99}, one verifies that the map $\theta\colon\raag{\shadow{\mathfrak M}}\to \bbgm{\mathfrak M}\rtimes_{\psi_v}\zz$
  defined  on generators by $\theta (w)=p(w,v)s$ is an inverse of $\widetilde \phi$.
    In particular, the map $\widetilde \phi$ is injective, and hence its restriction $\phi$ is also injective. This completes the proof of the connected case.

Next, we consider the case where $\mathfrak M$ is disconnected. Let $\mathfrak M_1,\dots,\mathfrak M_n$ be the connected components of $\mathfrak M$.
    Since each $\mathfrak M_i$ is graphic and connected, the map $\phi_i\colon\bbgm{\mathfrak M_i} \to \bbg{\shadow{\mathfrak M_i}}$ is an isomorphism.
Since every ordered signed circuit of $\mathfrak M$ is an ordered signed circuit of $\mathfrak M_i$ for some $i$, it follows from the presentation~\eqref{eq:OMG} that $\bbgm{\mathfrak M}\cong\Asterisk_{i=1}^n \bbgm{\mathfrak{M}_i}\cong\Asterisk_{i=1}^n\bbg{\shadow{\mathfrak{M}_i}}$. 
Finally, identifying one  vertex from each $\shadow{\mathfrak M_i}$ gives $\bigvee_{i=1}^n \shadow{\mathfrak M_i}$, and it follows from the Dicks--Leary presentation~\eqref{eq:infdicksleary} that $\bbg{\bigvee_{i=1}^n \shadow{\mathfrak M_i}}\cong\Asterisk_{i=1}^n\bbg{\shadow{\mathfrak{M}_i}}$. This completes the proof.
\end{proof}

We end this section with a proposition about the role of the ordering for a graphic matroid.
Let $M$ be a connected graphic matroid and consider the following sets:
\begin{itemize}
    \item $\Gamma(M)=\{\Gamma \mid \Gamma$  is a  biconnected graph whose cycle matroid is $M\}$.
    \item $\Omega_{gc}(M)=\{\omega \mid \omega$ is a globally compatible ordering on $M\}.$ 
\end{itemize}
Here, we say that $\omega$ is a globally compatible of $M$ if it is globally compatible with respect to some orientation of $M$. Note that by Lemma~\ref{lem:reorientation}, this depends only on the reorientation class, and graphic matroids  have a unique reorientation class by \cite[Proposition 6.2]{BLV78}.

\begin{proposition}\label{prop:defining graphs and orderings}
Let $M$ be a connected graphic matroid. There is a bijection $\alpha\colon\Gamma(M) \to \Omega_{gc}(M)$ defined by $\alpha(\Gamma)=\omega_\Gamma$.
\end{proposition}
\begin{proof}
    The inverse is the map $\eta\colon\Omega_{gc}(M)\to \Gamma(M)$ that associates a globally compatible ordering to the orderly graph $\shadow{\mathfrak M}$ of the orderly matroid $\mathfrak M$ obtained by equipping $M$ with $\omega$ and an arbitrary orientation.
    By Lemma~\ref{lem:orderly graph of a graph}, the composition $\eta\circ\alpha$ is the identity on $\Gamma(M)$.   
    Conversely, 
    if $\omega$ is a globally compatible ordering on $M$ satisfying Definition~\ref{def:globally compatible}, then it follows from Lemma~\ref{lem:faithful} that there is an isomorphism between $\mathfrak M$ and the orderly cycle matroid of $\shadow{ \mathfrak M}$ and $\omega=\omega_{\shadow{ \mathfrak M}}$. This implies that $\alpha \circ \eta$ is the identity on $\Omega_{gc}(M)$.
\end{proof}

\section{Tree clique-spanners}\label{sec:tcs}
In this section, we introduce a certain class of spanning trees.
After studying some examples and basic properties, we review their connections with dually chordal graphs in \S\ref{sec:dually chordal} and with locally connected spanning trees in \S\ref{sec:loc con}.

Let $\Gamma$ be a simplicial graph.
An \textit{$n$-clique} of $\Gamma$ is a complete subgraph $K_n$ on $n$ vertices; it corresponds to a simplex of dimension $n-1$ in $\flag \Gamma$. 
Given a spanning tree $T$ of $\Gamma$ and two vertices $u,v\in \vv \Gamma$, we denote by $[u,v]_T$ the unique edge-path in $T$ from $u$ to $v$, and by $ \dist Tuv$ and $\dist \Gamma uv$ the distances between $u$ and $v$ in $T$ and in $\Gamma$, respectively.

\begin{definition}[Tree clique-spanner]\label{def:TCS}
A spanning tree $T$ of $\Gamma$ is called a \textit{tree clique-spanner} if for every edge $\{u,v\}\in \ee \Gamma$, there exists a clique $K$ of $\Gamma$ such that $[u,v]_T\subseteq K$.
There is a unique minimal such clique, which we denote by $\suppclique Tuv$  and call the \textit{supporting clique} of $\{u,v\}$ with respect to $T$.
\end{definition}

It is known that $\Gamma$ admits a tree clique-spanner if and only if it is \textit{dually chordal}~\cite[Theorem 2.1]{SB94}; see the definition in the Introduction and the discussion in \S\ref{sec:dually chordal} below.
We now give some examples that are more meaningful for our main application in Theorem~\ref{mainthm:raag recognition}.

\begin{example}[Trees, cliques, and cones]\label{ex:TCS of tree & clique}
    A tree is a tree clique-spanner of itself. Any spanning tree of a clique is a tree clique-spanner of that clique. A cone graph admits a tree clique-spanner given by the spoke at a cone vertex. 
\end{example}

\begin{example}[Graphs without tree clique-spanners]\label{ex:no tcs}
A cycle $C_n$  admits a tree clique-spanner if and only if $n=3$. 
The trefoil graph in Figure~\ref{fig:trefoil with T3S} does not admit tree clique-spanners.
\end{example}

\begin{example}[Tree $k$-spanners]\label{ex:kspanners}
A spanning tree $T$ of $\Gamma$ is called a \textit{tree $k$-spanner} if every pair of vertices $u,v\in \vv \Gamma$ satisfies $ \dist Tuv \leq k \dist \Gamma uv.$
It follows from~\cite[Lemma 3.1]{CR26} that a tree $2$-spanner is a tree clique-spanner.
Conversely, any tree clique-spanner is a tree $k$-spanner, where $k+1$ is the number of vertices in the largest clique in $\Gamma$.
On the other hand, see the picture on the right in Figure~\ref{fig:TCS} for a tree clique-spanner that is not a tree $2$-spanner, and see Figure~\ref{fig:trefoil with T3S} for a tree $3$-spanner that is not a tree clique-spanner.
\end{example}

\begin{figure}[h]
    \centering
    \begin{tikzpicture}[scale=0.6]
\draw [thick] (-1,0)--(1,0);
\draw [thick] (0,2)--(1,0);
\draw [thick] (0,2)--(-1,0);
\draw [thick] (-1,0)--(0,-2);
\draw [thick] (-1,0)--(1,0);
\draw [thick] (0,-2)--(1,0);
\draw [thick] (1,0)--(2,-2);
\draw [thick] (2,-2)--(0,-2);
\draw [thick] (1,0)--(0,-2);
\draw [thick] (-1,0)--(-2,-2);
\draw [thick] (-2,-2)--(0,-2);
\draw [thick] (-1,0)--(0,-2);

\draw [thick, red] (0,2)--(1,0);
\draw [thick, red] (2,-2)--(0,-2);
\draw [thick, red] (1,0)--(0,-2);
\draw [thick, red] (-2,-2)--(0,-2);
\draw [thick, red] (-1,0)--(0,-2);

\draw [fill] (0,2) circle [radius=0.15];
\draw [fill] (-1,0) circle [radius=0.15];
\draw [fill] (1,0) circle [radius=0.15];
\draw [fill] (-2,-2) circle [radius=0.15];
\draw [fill] (0,-2) circle [radius=0.15];
\draw [fill] (2,-2) circle [radius=0.15];
\end{tikzpicture}
    \caption{The trefoil graph with a tree $3$-spanner colored in red.}
    \label{fig:trefoil with T3S}
\end{figure}
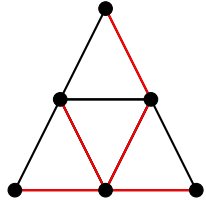

Our main motivation for defining tree clique-spanners is that they generalize tree 2-spanners, which we studied in \cite{CR26} in relation to the RAAG recognition problem for BBGs.
We now proceed to generalize the results in \cite[\S 3.1]{CR26} from the context of tree $2$-spanners to that of tree clique-spanners.

Let $v_1,\dots, v_n \in \vv\Gamma$. We denote by $[v_1,\dots,v_n]_T$ the \textit{convex hull} of $v_1,\dots, v_n $ in $T$, that is, the union of the paths in $T$ between every pair of these vertices.
In particular, the convex hull $[v_i,v_j]_T$ is the unique path in $T$ between $v_i$ and $v_j$.

\begin{lemma}\label{lem:0simplexinclique}
    Let $\Gamma$ be a  simplicial graph.
    Let $T$ be a tree clique-spanner of $\Gamma$, and
    let $K$ be a clique of $\Gamma$ with vertices $\{v_1,\dots,v_n\}$.
    Then the subgraph induced by the convex hull  $[v_1,\dots,v_n]_T$ is a clique.
    In particular, this clique contains $K$.
\end{lemma}

\begin{proof}
Let $p$ and $q$ be two vertices of $[v_1,\dots,v_n]_T$.
Since $[v_1,\dots,v_n]_T$ is convex in $T$, the  geodesic $[p,q]_T$ is contained in $[v_1,\dots,v_n]_T$.
Moreover, since $[v_1,\dots,v_n]_T$ is the convex hull of the vertices of $K$ in $T$, every leaf of  $[v_1,\dots,v_n]_T$ is a vertex of $K$.
Thus, the geodesic $[p,q]_T$ is contained in $[v_i,v_j]_T$ for some $1\leq i,j\leq n$.
Since $v_i$ and $v_j$ are adjacent in $\Gamma$ and $T$ is a tree clique-spanner of $\Gamma$, the path $[v_i,v_j]_T$ is contained in a clique of $\Gamma$. Therefore, the vertices $p$ and $q$ are adjacent in $\Gamma$.
\end{proof}

\begin{definition}[Supporting clique]\label{def:supp_clique}
Given a clique $K$ with vertices $\{v_1,\dots,v_n\}$ and a tree clique-spanner $T$ of $\Gamma$, we denote by $\suppcliqueclique TK$ the clique in $\Gamma$ spanned by $[v_1,\dots,v_n]_T$; see Lemma~\ref{lem:0simplexinclique}.
We call it the \textit{supporting clique} of $K$, extending the terminology from Definition~\ref{def:TCS}.
Note that this is the unique minimal clique of $\Gamma$ containing $[v_1,\dots,v_n]_T$.
\end{definition}

For example, let $\Lambda$ be the trefoil graph in Figure~\ref{fig:trefoil with T3S}, and let $\Gamma=\{v\} \ast \Lambda$ be the cone over the trefoil graph. If $K$ is the central triangle of $\Lambda$ and $T$ is the spoke at $v$, then $\suppcliqueclique TK=\{v\} \ast K$.

While it is known that  the graphs that admit tree clique-spanners are exactly the dually chordal graphs, we would like conditions that are more elementary to check than being dually chordal.
The following result generalizes \cite[Lemma 3.4]{CR26} and provides a necessary topological condition, rather than a combinatorial one, for a graph to admit a tree clique-spanner.

\begin{lemma}\label{lem: simply connected}
Let $\Gamma$ be a  simplicial graph.
If $\Gamma$ admits a tree clique-spanner, then the flag complex $\flag\Gamma$ is simply connected.
\end{lemma}
\begin{proof}
Let $C=(e_1,\dots,e_n)$ be a cycle of $\Gamma$, where $e_i=\{v_i,v_{i+1}\}$ with indices taken mod $n$ are the edges of $C$, and
let $T$ be a tree clique-spanner of $\Gamma$.
For each edge $e_i=\{v_i,v_{i+1}\}$, let $K_i = \suppclique{T}{v_i}{v_{i+1}}$ and $\gamma_i=[v_i,v_{i+1}]_T$.
By definition, we have $\gamma_i\subseteq K_i$.
Let $L$ be the closed edge-path obtained by concatenating the paths $\gamma_1,\dots,\gamma_n$.
Since $K_i$ is a clique, we have that $\pi_1(\flag{K_i})=1$. Thus, the edge $e_i$ and the path $\gamma_i$ are homotopic in $\flag{K_i}$ relative to their endpoints.
Therefore, the closed edge-path $L$ is homotopic to $C$ in $\flag \Gamma$.
Since $L$ is contained in the tree clique-spanner $T$ by construction, it is null-homotopic. Thus,  $C$ is also null-homotopic. Hence, the flag complex $\flag\Gamma$ is simply connected.
\end{proof}

\begin{remark}\label{rmk:contractible}
    The proof of Lemma~\ref{lem: simply connected} also proves that if $\Gamma$ has a spanning tree $T$ such that, for every edge $\{u,v\}\in \ee \Gamma$, the  path $[u,v]_T$ is contained in a simply connected full subcomplex of $\flag \Gamma$, then $\flag \Gamma$ is simply connected. 
    When $T$ is a tree clique-spanner of $\Gamma$, we actually obtain, via group-theoretic techniques, that $\flag \Gamma$ must be contractible; see   Corollary~\ref{cor:contractible}.
\end{remark}

\subsection{Dually chordal graphs}\label{sec:dually chordal}
Graphs admitting tree clique-spanners have already been studied in the literature with different names, such as \textit{expanded trees}~\cite{SB94}, \textit{tree-clique graphs}~\cite{GO96}, and \textit{dually chordal graphs}~\cite{BDCV98,duallychordal2012}. For the reader's convenience, we collect various characterizations in Theorem~\ref{thm:dually chordal}.

Let $\Gamma$ be a simplicial graph, and let $u,v \in  \vv \Gamma$. An $(u,v)$-\emph{separator} is a subset $S\subseteq V(\Gamma)$ such that removing $S$ disconnects $u$ and $v$. 
A \emph{minimal $(u,v)$-separator} is a $(u,v)$-separator whose proper subsets are not $(u,v)$-separators. 
A \textit{minimal separator} is a minimal $(u,v)$-separator for some $u,v \in  \vv \Gamma$. 
We denote by $\induced \Gamma S$ the full subgraph of $\Gamma$ induced by $S$.
We say that a minimal separator $S$ is \textit{connected} if $\induced \Gamma S$ is a connected graph.

\begin{theorem}\label{thm:dually chordal}
Let $\Gamma$ be a simplicial graph, and let $T$ be a spanning tree of $\Gamma$. 
The following statements are equivalent.
\begin{enumerate}
    \item \label{item:dually chordal tcs} $T$ is a tree clique-spanner of $\Gamma$.
    \item \label{item:dually chordal max clique} For every maximal clique $K$ of $\Gamma$, the intersection $ T\cap K$ is a tree (equivalently, it is connected).
    \item \label{item:dually chordal minsep} For every minimal separator $S$ of $\Gamma$, the intersection $T\cap \induced \Gamma S$ 
    is a  tree (equivalently, it is connected).
\end{enumerate}
Moreover, a graph $\Gamma$ admits a spanning tree with one of the above properties if and only if it is a dually chordal graph.
\end{theorem}

\begin{proof}
First, note that a tree satisfying \eqref{item:dually chordal max clique} is called a \textit{compatible tree} in \cite{GO96,duallychordal2012}.
Let $T$ be a tree clique-spanner of $\Gamma$, and let $K$ be a maximal $n$-clique whose vertices are $v_1,\dots,v_n$. 
By Lemma~\ref{lem:0simplexinclique},
the supporting clique $\suppcliqueclique{T}{K}$ contains the convex hull $[v_1,\dots,v_n]_T$.
Since $K$ is maximal, we have $K=\suppcliqueclique{T}{K}$. Hence, $[v_1,\dots,v_n]_T \subseteq K$.
In particular, intersection $T\cap K= [v_1,\dots,v_n]_T$ is a subtree of $T$ by definition.
This proves  \eqref{item:dually chordal tcs} $\Rightarrow$ \eqref{item:dually chordal max clique}. 
For the implication \eqref{item:dually chordal max clique} $\Rightarrow$ \eqref{item:dually chordal tcs} see \cite[Lemma 3.3]{GO96}.
The implication \eqref{item:dually chordal max clique} $\Rightarrow$ \eqref{item:dually chordal minsep} follows from \cite[Theorem 2]{duallychordal2012}, and \eqref{item:dually chordal minsep} $\Rightarrow$ \eqref{item:dually chordal max clique} follows from the proof of \cite[Theorem 5]{duallychordal2012}. 
    
Finally, by \cite[Theorem 2.1]{SB94}, a graph $\Gamma$ admits a tree clique-spanner if and only if it is a dually chordal graph. In \cite{SB94}, an \textit{expanded tree} is a graph admitting a tree clique-spanner, and a tree clique-spanner is called a \textit{canonical tree}. 
\end{proof}

Next, we use the characterizations in Theorem~\ref{thm:dually chordal} to obtain two easy-to-check conditions for the existence of certain spanning trees.
For instance, both conditions apply to the graph on the right-hand side of Figure~\ref{fig:TCS}.

\begin{lemma}\label{lem:disjoint_minsep_tcs}
    Let $\Gamma$ be a connected simplicial graph whose  minimal separators are connected and pairwise disjoint. 
    Then $\Gamma$ admits a tree clique-spanner.
\end{lemma}
\begin{proof}
    Choose a spanning tree $T_S$ of $\induced\Gamma S$ for each minimal separator $S$. 
    Since the minimal separators are pairwise disjoint, the union $F=\bigcup_S T_S$ is a forest in $\Gamma$.
    Extend $F$ to a spanning tree $T$ of $\Gamma$.
    We claim that $T \cap \induced \Gamma S=T_S$. The inclusion $T_S\subseteq T\cap\Gamma_S$ follows from the construction.
    For the converse, let $e\in\ee{T\cap\Gamma_S}$. If $e\notin\ee{T_S}$, then $\{e\}\cup T_S\subseteq\Gamma_S$ contains a cycle in $\Gamma_S$. This is absurd since $\{e\}\cup T_S\subseteq T$. This proves the claim.  
    The lemma now follows from Theorem~\ref{thm:dually chordal}~\eqref{item:dually chordal minsep}.
\end{proof}

\begin{lemma}\label{lem:no t2s}
Let $\Gamma$ be a biconnected simplicial graph. Suppose that there is a maximal clique $K\subseteq \Gamma$ with disjoint minimal separators $S_1,S_2\subseteq \vv K$ of $\Gamma$. Then $\Gamma$ does not admit any tree 2-spanner.
\end{lemma}
\begin{proof}
Suppose by contradiction that $\Gamma$ admits a tree 2-spanner $T$.
Note that $T$ is a tree clique-spanner; see Example~\ref{ex:kspanners}. Thus, by Theorem~\ref{thm:dually chordal}~\eqref{item:dually chordal minsep}, the intersection $T \cap \induced{\Gamma}{S_i}$ is a tree. 
Since $\Gamma$ is biconnected, both $S_1$ and $S_2$ contain at least one edge.
Therefore, we can choose an edge $e_i=\{u_i,v_i\}\in \ee{T \cap \induced{\Gamma}{S_i}}$.

Since $S_1\cap S_2=\varnothing$ and $K$ is a clique, the edges $e_1$ and $e_2$ span a $K_4$ inside $K$. Notice that at most one edge joining an endpoint of $e_1$ to an endpoint of $e_2$ is in $\ee T$. If such an edge exists, say $\{u_1,u_2\}\in\ee T$, then $\dist{T}{v_1}{v_2}=3>2=2\dist{\Gamma}{v_1}{v_2}$, contradicting the assumption that $T$ is a tree 2-spanner. 
If no edge joining an endpoint of $e_1$ to an endpoint of $e_2$ belongs to $\ee T$, then $\dist{T}{v_1}{v_2}=2$, since $T$ is a tree 2-spanner and $\{v_1,v_2\}\in\ee K$. Let $P$ be the edge-path in $T$ between $v_1$ and $v_2$. Then $P$ does not contain $e_1$ and $e_2$. Similarly, the edge-path $Q$ in $T$ between $u_1$ and $u_2$ also avoids $e_1$ and $e_2$. Now, the union $P\cup Q\cup\{e_1,e_2\}$ contains a cycle in $T$, which contradicts the fact that $T$ is a spanning tree. Therefore, the graph $\Gamma$ does not admit any tree 2-spanner.
\end{proof}

\subsection{Locally connected spanning trees}\label{sec:loc con}
In this section, we explore an additional property of tree clique-spanners. Although this property is not needed for the group-theoretic applications developed below, it may be of independent interest.

We say that a spanning tree $T$ of $\Gamma$ is \textit{locally connected} if, for each vertex $v\in V(T)$, the subgraph of $\Gamma$ induced by $\lk{v}{T}$ is connected. 
Notice that if $\Gamma$ has a cut vertex, then no spanning tree of $\Gamma$ is locally connected.
Cai proved in \cite[Corollary 1.2]{Cai1997} that every tree $2$-spanner of a $2$-connected graph is locally connected. 
In the next proposition, we generalize Cai's result to tree clique-spanners of biconnected graphs.

\begin{proposition}\label{prop:loc_con}
Let $\Gamma$ be a biconnected simplicial graph, and let $T$ be a tree clique-spanner of $\Gamma$.
Then $T$ is locally connected.
\end{proposition}

\begin{proof}
Let $T$ be a tree clique-spanner of $\Gamma$, and let $v\in \vv\Gamma$.
Since $v$ is not a cut vertex and $\flag \Gamma$ is simply connected by Lemma~\ref{lem: simply connected}, the induced subgraph $\induced{\Gamma}{\lk v\Gamma}$ is connected; see \cite[Lemma 4.17 (2)]{CR26}.

For each connected component $C$ of $T\setminus\{v\}$, there is a unique vertex $r_C\in\vv C$ adjacent to $v$ in $\Gamma$. Let $C$ and $D$ be connected components of $T\setminus\{v\}$. Since $\induced{\Gamma}{\lk v\Gamma}$ is connected, there is a veretx-path $(r_C=p_0,p_1,\dots,p_n=r_D)$ in $\induced{\Gamma}{\lk v\Gamma}$, with $p_{i-1}$ adjacent to $p_i$. 
For each $0\leq i\leq n$, let $C_i$ be the connected component of $T\setminus\{v\}$ containing $p_i$. Then $C=C_0$ and $D=C_n$.

Consider $p_{i-1}$ and $p_i$. If $C_{i-1}=C_i$, then $r_{C_{i-1}}=r_{C_i}$. If $C_{i-1}\neq C_i$, then the path $[p_{i-1},p_i]_T$ must contain $r_{C_{i-1}}$ and $r_{C_i}$. 
Since $p_{i-1}$ and $p_i$ are adjacent in $\Gamma$ and $T$ is a tree clique-spanner of $\Gamma$, the path $[p_{i-1},p_i]_T$ is contained in a clique of $\Gamma$. Thus, the vertices $r_{C_{i-1}}$ and $r_{C_i}$ are adjacent in $\Gamma$.

We now obtain a sequence of vertices $r_C=r_{C_0},r_{C_1},\dots,r_{C_n}=r_D$, where two consecutive vertices are either identical or adjacent in $\Gamma$. If two consecutive vertices are identical, delete one of them. If a vertex occurs more than once but not consecutively, say $r_{C_j}=r_{C_k}$ for some $j<k$, then delete $r_{C_{j+1}},\dots,r_{C_k}$. Repeating this process produces a vertex-path from $r_C=r_{C_0}$ to $r_D=r_{C_n}$ in $\Gamma$. Since $r_{C_i}\in\lk vT$ for all $0\leq i\leq n$, this path is in $\induced{\Gamma}{\lk vT}$. Thus, the induced subgraph $\induced{\Gamma}{\lk vT}$ is connected. Hence, the tree clique-spanner $T$ is locally connected. 
\end{proof}

The following example shows that the converse of Proposition~\ref{prop:loc_con} does not hold.

\begin{example}\label{ex:mainex}
    The spanning tree $T$ (red) of the graph $\Gamma$ shown in
    Figure~\ref{fig: loc_con_no_clique} is locally connected.
    It is not a tree clique-spanner because the path $[u,v]_T$ is not contained in any clique of $\Gamma$.
    Moreover, the group $\bbg \Gamma$ is not a RAAG by \cite[Example 5.17]{CR26}, and it follows from Theorem~\ref{mainthm:raag recognition} that $\Gamma$ does not admit any tree clique-spanner. 
    This shows that \cite[Proposition 3.2]{Cai1997} does not generalize directly to the setting of higher-dimensional $k$-trees and tree clique-spanners.
\end{example}

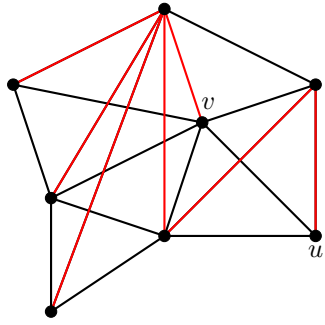
\begin{figure}[ht!]
    \centering    \begin{tikzpicture}[scale=0.5]
\draw [thick] (3,3)--(6,2); 
\draw [thick] (6,2)--(7,5);
\draw [thick] (7,5)--(3,3);

\draw [thick] (3,3)--(2,6);
\draw [thick] (2,6)--(7,5);

\draw [thick] (6,2)--(10,6);
\draw [thick] (10,6)--(7,5);

\draw [thick] (3,3)--(3,0)--(6,2);

\draw [thick] (6,2)--(10,2);
\draw [thick] (10,2)--(10,6);
\draw [thick] (10,2)--(7,5);

\draw [thick] (6,8)--(2,6);
\draw [thick] (6,8)--(3,3);
\draw [thick] (6,8)--(3,0);
\draw [thick, red] (6,8)--(6,2);
\draw [thick, red] (6,8)--(7,5);
\draw [thick] (6,8)--(10,6);

\node [below] at (10,2) {$u$};
\node [above] at (7.16,5.1) {$v$};

\draw [thick, red] (6,8)--(2,6);
\draw [thick, red] (6,8)--(3,3);
\draw [thick, red] (6,8)--(3,0);
\draw [thick, red] (6,2)--(10,6)--(10,2);

\draw [fill] (6,8) circle [radius=0.15];
\draw [fill] (2,6) circle [radius=0.15];
\draw [fill] (3,3) circle [radius=0.15];
\draw [fill] (3,0) circle [radius=0.15];
\draw [fill] (6,2) circle [radius=0.15];
\draw [fill] (7,5) circle [radius=0.15];
\draw [fill] (10,6) circle [radius=0.15];
\draw [fill] (10,2) circle [radius=0.15];

\end{tikzpicture}
    \caption{A graph with a locally connected spanning tree and without tree clique-spanners.}
    \label{fig: loc_con_no_clique}
\end{figure}

Proposition~\ref{prop:loc_con} has two corollaries.
The first one generalizes a result of Lin--Chang--Chen for strongly chordal graphs; see \cite[Corollary 5]{LCC07}.

\begin{corollary}\label{cor:loc con dually chordal}
    Biconnected dually chordal graphs admit locally connected spanning trees.
\end{corollary}
\begin{proof}
    Let $\Gamma$ be a biconnected dually chordal graph. Then $\Gamma$ admits a tree clique-spanner $T$ by Theorem~\ref{thm:dually chordal}, and $T$ is locally connected by Proposition~\ref{prop:loc_con}.
\end{proof}

\begin{corollary}
    Let $\Gamma$ be biconnected  and dually chordal.
    Then $\Gamma$ admits a trefoil-free spanning 2-tree.
\end{corollary}
\begin{proof}
    This follows from Proposition~\ref{prop:loc_con} and \cite[Theorem 1.1]{Cai1997}.
\end{proof}

\section{Isomorphism problems for BBGs}\label{sec:isoproblems}
We now apply the theory of tree clique-spanners developed above to the two isomorphism problems for BBGs mentioned in the Introduction. 

\subsection{BBGs on dually chordal graphs are RAAGs}
In this section, we deal with Problem~\ref{problem raag}.

\begin{definition}[Dual graph]\label{def:dual_graph}
Let $\Gamma$ be a simplicial graph, and let $T$ be a tree clique-spanner of $\Gamma$.
The \textit{dual graph} of $T$ is 
the graph $\dualtree T$ whose vertices are the edges of $T$, and two vertices are adjacent if and only if the corresponding edges of $T$ are contained in the same clique of $\Gamma$.
\end{definition}

Note that $\dualtree T$ is not defined as a subgraph of $\Gamma$.
We will prove in Corollary~\ref{cor:combinatorial applications} that any two tree clique-spanners in $\Gamma$ have isomorphic dual graphs.

We now show prove Theorem~\ref{mainthm:raag recognition}, by showing that a tree clique-spanner $T$ of $\Gamma$ can be used to simplify  the Dicks--Leary presentation \eqref{eq:dicksleary} of $\bbg \Gamma$  to the RAAG presentation of the RAAG on $\dualtree T$; see  \eqref{eq:bbg tcs raag} below.

\begin{proof}[Proof of Theorem~\ref{mainthm:raag recognition}]
    By Lemma~\ref{lem: simply connected}, the flag complex $\flag\Gamma$ is simply connected, so $\bbg \Gamma$ is finitely presented by \cite{BB1997}. Fix an orientation of the edges of $\Gamma$,  and consider the finite presentation \eqref{eq:dicksleary} from \cite{DicksLeary99}.
    
    Let $T$ be a tree clique-spanner of $\Gamma$.
    We partition the set of unoriented edges $\ee \Gamma$ into $\ee T$ and its complement $\ee\Gamma \setminus \ee T$, and we partition the generating set $\ee \Gamma ^\pm$ accordingly into 
    $$\ee T^\pm=\{t,\bar t\in \ee\Gamma^\pm \mid t\in \ee T\} \textrm{\quad and \quad} \ee\Gamma^\pm \setminus \ee T^\pm.$$

    Our strategy is to use $T$ to reduce the generating set from $\ee \Gamma ^\pm$ to $\ee T ^\pm$ by expressing the edges not in $T$ as paths in $T$.
    Consider the following (infinite) collections of words in $\ee \Gamma^\pm$:
    \begin{enumerate}
        \item $C_T$ consists of commutators of the form $[t_i,t_j]$, where $t_i$ and $t_j$ are edges in $T$ that are contained in the same clique of $\Gamma$.
        \item $L_T$ consists of \textit{long} relators of the form $e=t_{1}^{\epsilon_1}\cdots t_{{d}}^{\epsilon_{d}}$, where $e \not \in \ee T^\pm$, $d\in \nn$, $t_{i}\in \ee T^\pm$ for $i=1,\dots,d$,  $(\bar e,t_{1},\dots,t_{{d}})$ is an oriented cycle in $\Gamma$, and $\epsilon_i = \pm 1$  is chosen according to the fixed orientation of $\Gamma$. 
    \end{enumerate}
    Since the BBG defined by a clique is abelian, two edges commute if they appear in the same clique.
    Moreover, $L_T\subsetneq R_\Gamma$. 
    Thus, all the  words in $C_T$ and $L_T$ represent the trivial element in $\bbg \Gamma$. Therefore, they can be added as redundant relators to the Dicks--Leary presentation \eqref{eq:dicksleary}, giving the following alternative presentation:
    \begin{equation}\label{eq:dicksleary2}
    \bbg \Gamma =  \langle \ee\Gamma^\pm \mid I_\Gamma,  R^\Delta_\Gamma,C_T,L_T  \rangle.
    \end{equation}

    Let $\tau = (v_1,v_2,v_3)$ be an oriented triangle with vertices $v_i$, and denote its edges by $f_1=(v_2,v_3)$, $f_2=(v_3,v_1)$, and $f_3=(v_1,v_2)$.
    Let $\suppcliqueclique T\tau$ be the supporting clique of $\tau$ from Definition~\ref{def:supp_clique}.    
    The triangle $\tau$ gives rise to two relators $f_1f_2f_3$ and $f_3f_2f_1$ in $R^\Delta_\Gamma$.
    We claim that both relators follow from the relators in $C_T$ and $L_T$. Note that for an oriented edge $f_k=(v_i,v_j)$, we have the following equality in $\bbg\Gamma$: 
    $$f_k= t_1^{\epsilon_1}\cdots t_p^{\epsilon_p},
    $$ 
    where $p\in \nn$, $t_1,\dots,t_p$ are the edges along $[v_i,v_j]_T$, and $\epsilon_i=\pm 1$ is chosen according to the orientation of $\Gamma$.
    This is clear if $f_k\in \ee T ^\pm$ (just take $p=1$), while for $f_k\in \ee\Gamma^\pm \setminus \ee T^\pm$, we use the long relators from $L_T$.
    
    It follows that $f_1f_2f_3$ can be written as a product of elements of $\ee T ^\pm$ corresponding to the edges of $T$ along the edge-loop $\gamma=[v_2,v_3]_T [v_3,v_1]_T[v_1,v_2]_T$.
    Note that $\gamma $ is contained in $[v_1,v_2,v_3]_T$ and, therefore, in the supporting clique $\suppcliqueclique T\tau$. In particular, all the edges that appear in $\gamma$ commute.
    Moreover, since $[v_1,v_2,v_3]_T$ is a subdivision of a (possibly degenerate) tripod, all the edges along $\gamma$ appear in pairs with opposite orientations.
    Thus, everything cancels out in the expression of $f_1f_2f_3$ as a product of elements of $\ee T ^\pm$.
    The same argument applies to the other relator $f_3f_2f_1$ associated with $\tau$. This proves the claim.
  
    This shows that we can remove $R^\Delta_\Gamma$ from the presentation \eqref{eq:dicksleary2}. Thus, we obtain the following presentation:
\begin{equation}\label{eq:dicksleary3}
    \bbg \Gamma =  \langle \ee\Gamma^\pm \mid I_\Gamma,C_T,L_T  \rangle.
    \end{equation}
    
    Finally, 
    since every generator from $\ee\Gamma^\pm \setminus \ee T^\pm$ appears as a single letter in a relator from $L_T$,
    we can simultaneously drop $\ee\Gamma^\pm \setminus \ee T^\pm$ and $L_T$.
    We also drop all relators $I_\Gamma$ and all the generators of the form $\bar e \in \ee T^-$.
    What is left is the presentation
    \begin{equation}\label{eq:bbg tcs raag}
    \bbg \Gamma =\langle \ee T \mid C_T\rangle,
    \end{equation}
    which is a RAAG presentation. 
    Moreover, it follows directly from the definitions of $\dualtree T$ and $C_T$ that this is a presentation of the RAAG $\raag{\dualtree T}$, where $\dualtree T$ is the dual graph of $T$.
\end{proof}

As an application, we can recognize many graphs whose associated BBGs are RAAGs.

\begin{proof}[Proof of Corollary~\ref{cor:more bbgs}]
For each of the  cases, we check that $\Gamma$ admits a tree clique-spanner, and then apply Theorem~\ref{mainthm:raag recognition} to conclude that $\bbg \Gamma$ is a RAAG.

If $\Gamma$ is dually chordal, then it admits a tree clique-spanner by  Theorem~\ref{thm:dually chordal}.
If $\Gamma$ is strongly chordal, then it is dually chordal by \cite[Corollary 3]{BDCV98}. 
Finally, if $\Gamma$ has connected and pairwise disjoint minimal separators, then $\Gamma$  admits a tree clique-spanner by Lemma~\ref{lem:disjoint_minsep_tcs}.
\end{proof}

\begin{remark}\label{rmk:new}
It is not difficult to construct graphs that satisfy one of the conditions in Corollary~\ref{cor:more bbgs} and the hypotheses of Lemma~\ref{lem:no t2s}; see the right picture in Figure~\ref{fig:TCS}. Such graphs admit tree clique-spanners but do not admit tree 2-spanners.
Thus, Corollary~\ref{cor:more bbgs} provides infinitely many new examples of BBGs that were not previously known to be RAAGs.
\end{remark}

We now present some combinatorial applications of Theorem~\ref{mainthm:raag recognition} that generalize results on tree 2-spanners from \cite[\S 3]{CR26}.
Recall that non-isomorphic graphs can have isomorphic BBGs; see Proposition~\ref{free} and Proposition~\ref{prop:2isobbg}.
Since a graph is dually chordal if and only if it admits a tree clique-spanner (Theorem~\ref{thm:dually chordal}), the following result provides a solution to the graph isomorphism problem (Problem~\ref{problem graph}) for BBGs on dually chordal graphs.

\begin{corollary}\label{cor:combinatorial applications}
The following statements hold.
    \begin{enumerate}
    \item \label{item:isobbg} Let $\Gamma$ and $\Lambda$ be two simplicial graphs admitting tree clique-spanners $T_\Gamma$ and $T_\Lambda$,
    respectively. Then $\bbg \Gamma \cong \bbg \Lambda$  if and only if $\dualtree{T_\Gamma} \cong \dualtree{T_\Lambda}$.

    \item \label{item:isodual} If $T_1$ and $T_2$ are tree clique-spanners of $\Gamma$, then  $T_1^\ast\cong T_2^\ast$. 
    \end{enumerate}
\end{corollary}
\begin{proof}
    By Theorem~\ref{mainthm:raag recognition}, we have 
    $\bbg \Gamma \cong \raag{T_\Gamma^*}$ and $ \bbg \Lambda \cong \raag{T_\Lambda^*}$.
    Then \eqref{item:isobbg} follows from the fact that two RAAGs are isomorphic if and only if their defining graphs are isomorphic \cite{DromsIsomorphismsofGraphGroups}. 
    Taking $\Lambda=\Gamma$ in \eqref{item:isobbg} proves \eqref{item:isodual}.
\end{proof}

The following may be of independent interest, and it is stated for a possibly infinite flag simplicial complex $\Delta$. Of course, it includes the case that $\Delta$ itself is compact and admits a tree clique-spanner.

\begin{corollary}\label{cor:contractible}
    If $\Delta$ is a  flag simplicial complex such that every compact subcomplex is contained in a compact subcomplex whose 1-skeleton admits a tree clique-spanner, then $\Delta$ is contractible.
\end{corollary}
\begin{proof}
    Let $f\colon S^n\to \Delta$ be a continuous map for some $n\geq 0$.
    Since $S^n$ is compact, the image of $f$ is contained in a compact subcomplex $K$ of $\Delta$. By assumption, the compact subcomplex $K$ is contained in a compact subcomplex $K'$ of $\Delta$ whose 1-skeleton $\Gamma$ admits a tree clique-spanner. 
    Since $\Delta$ is a flag simplicial complex, the flag complex $\flag\Gamma$ is contained in $\Delta$ and contains $K'$. In particular, it contains $f(S^n)$. 
    Note that the case where $n=0$ implies that $\Delta$ is path-connected, since $\Gamma$ is connected.

    By Theorem~\ref{mainthm:raag recognition}, the group $\bbg  \Gamma$ is a RAAG, and hence a group of finite type. It follows from \cite{BB1997} that $\flag \Gamma$ is acyclic. 
    Since $\flag \Gamma$ is also simply connected by Lemma~\ref{lem: simply connected}, it follows from the  Hurewicz theorem that $\pi_n(\flag\Gamma)=0$ for $n\geq 1$.
    Thus, the map $f$ is nullhomotopic. Hence, we have $\pi_n(\Delta)=0$ for $n\geq1$.
    By the Whitehead theorem, we can conclude that $\Delta$ is contractible.
\end{proof}

\subsection{Isomorphic BBGs on graphs with non-isomorphic matroids}\label{sec:non isom bbgs}
In this section, we deal with Problem~\ref{problem graph}.
Recall that a graph is \emph{$k$-connected}  if it has more than $k$ vertices, and removing fewer than $k$ vertices does not disconnect the graph.
For instance, being $1$-connected is equivalent to being connected, and a $2$-connected graph is also biconnected; that is, it is connected and has no cut vertices. 

\begin{lemma}\label{lem: dual graph of chordal is chordal}
    Let $\Gamma$ be a $k$-connected chordal graph with a tree clique-spanner $T$. Then $\dualtree T$ is a $(k-1)$-connected chordal graph.
\end{lemma}

\begin{proof}
Since $\Gamma$ is a $k$-connected chordal graph, every minimal separator of $\Gamma$ is a clique with at least $k$ vertices; see \cite{Dirac1961}. 

Let $K$ be a maximal clique of $\Gamma$. By Theorem~\ref{thm:dually chordal}~\eqref{item:dually chordal max clique}, the intersection $T\cap K$ is a spanning tree of $K$. Thus, the dual graph $\dualtree{(T\cap K)}$ is a clique with $|\vv K|-1$ vertices. 

Let $K_1$ and $K_2$ be maximal cliques of $\Gamma$ such that $S=\vv{K_1}\cap\vv{K_2}$ is a minimal separator of $\Gamma$. Then $\dualtree{(T\cap K_1)}\cap\dualtree{(T\cap K_2)}=\dualtree{(T \cap \induced{\Gamma}{S})}$. Since $T$ is a tree clique-spanner, it follows from Theorem~\ref{thm:dually chordal}~\eqref{item:dually chordal minsep} that $T \cap \induced{\Gamma}{S}$ is a spanning tree of $\induced{\Gamma}{S}$. Since $\induced{\Gamma}{S}$ is a clique, the dual graph $\dualtree{(T \cap \induced{\Gamma}{S})}$ is a clique with $|S|-1$ vertices.

It follows that $\dualtree T$ is obtained by gluing the cliques $\dualtree{(T\cap K_i)}$ as  $K_i$  ranges over the maximal cliques of $\Gamma$, along the cliques $\dualtree{(T \cap \induced{\Gamma}{S})}$ as $S$ ranges over the minimal separators of $\Gamma$. 
Therefore, the dual graph $\dualtree T$ is also chordal. 
Since $|S|\geq k$ for each minimal separator $S$ of $\Gamma$, the minimal separator $\dualtree{(T \cap \induced{\Gamma}{S})}$ has at least $|S|-1\geq k-1$ vertices. Hence, the graph $\dualtree T$ is $(k-1)$-connected. This completes the proof. 
\end{proof}

We are now ready to prove Theorem~\ref{introthm non isom}, which  gives a family of examples of isomorphic BBGs defined by non-isomorphic $k$-connected graphs.
This was previously known for $k=1$ (see Proposition~\ref{free}) and $k=2$ (see \cite[Example 3.7]{CR26} and Figure~\ref{fig:noniso_2connected}).

\begin{proof}[Proof of Theorem~\ref{introthm non isom}]
Let $\Gamma$ be a $k$-connected chordal graph with at least two minimal separators such that all the minimal separators are pairwise disjoint.
For concreteness, such $\Gamma$ can be obtained by gluing cliques along disjoint subcliques with at least $k$ vertices; see the right-hand side of Figure~\ref{fig:TCS} for $k=2$.

Since $\Gamma$ is chordal, its minimal separators are cliques, and therefore connected.
Then $\Gamma$ admits a tree clique-spanner $T$ by Lemma~\ref{lem:disjoint_minsep_tcs}, and it follows from Theorem~\ref{mainthm:raag recognition} that $\bbg\Gamma\cong\raag{\dualtree T}$, where $\dualtree T$ is the dual graph of $T$.
Let $\Lambda$ be the cone over $\dualtree T$. Then $\bbg\Lambda\cong\raag{\dualtree T}$; see \cite[Corollary 3.10]{CR26}. Therefore, $\bbg\Gamma\cong\bbg\Lambda$. By Theorem~\ref{mainthm matroid intro}, this is equivalent to saying that $\bbgm{\mathfrak M_\Gamma} \cong \bbgm{\mathfrak M_\Lambda}$.

It follows from Lemma~\ref{lem: dual graph of chordal is chordal} that $\dualtree T$ is $(k-1)$-connected. Thus, the graph $\Lambda$ is $k$-connected.
However, these two graphs $\Gamma$ and $\Lambda$ are not isomorphic because $\Lambda$ is a cone graph whereas $\Gamma$ is not, since $\Gamma$ has at least two disjoint minimal separators. 

Finally, note that two $3$-connected graphs are isomorphic if and only if they have isomorphic cycle matroids (see \cite[Theorem 5.3.1]{Oxley}). Therefore, we have $\mathfrak M_\Gamma\not \cong \mathfrak M_\Lambda$ for $k\geq 3$.
\end{proof}

\begin{remark}\label{rmk:final}
Some remarks about the proof of Theorem~\ref{introthm non isom} are in order.
\begin{enumerate}

    \item A key ingredient in the proof of Theorem~\ref{introthm non isom} is that the BBGs involved in the construction are all isomorphic to RAAGs.
    Note that for $k\geq2$, the graph $\Gamma$ in the proof satisfies the assumptions of Lemma~\ref{lem:no t2s}, so it does not admit a tree 2-spanner.
    Thus, the theory of tree clique-spanners developed above is really needed here.   

    \item The $k$-connectedness of a graph $\Gamma$ is not related to the finiteness properties of $\bbg \Gamma$. Indeed, the graphs used in Theorem~\ref{introthm non isom} are chordal, and in particular, their flag complexes are contractible.
    Thus, it follows from \cite{BB1997} that their BBGs are finitely presented and of type FP($R$) for any ring $R$ with $0\neq 1$. Moreover, they are coherent by \cite{BL25}.
\end{enumerate}
\end{remark}

\printbibliography

@article {CR26,
    AUTHOR = {Chang, Yu-Chan and Ruffoni, Lorenzo},
     TITLE = {A graphical description of the {BNS}-invariants of
              {B}estvina-{B}rady groups and the {RAAG} recognition problem},
   JOURNAL = {Groups Geom. Dyn.},
  FJOURNAL = {Groups, Geometry, and Dynamics},
    VOLUME = {20},
      YEAR = {2026},
    NUMBER = {2},
     PAGES = {487--539},
      ISSN = {1661-7207,1661-7215},
   MRCLASS = {20F36 (20F05 20F65 20J05)},
  MRNUMBER = {5047111},
       DOI = {10.4171/ggd/823},
       URL = {https://doi-org.proxy.lib.ohio-state.edu/10.4171/ggd/823},
}

@article {CH07,
    AUTHOR = {Charney, Ruth},
     TITLE = {An introduction to right-angled {A}rtin groups},
   JOURNAL = {Geom. Dedicata},
  FJOURNAL = {Geometriae Dedicata},
    VOLUME = {125},
      YEAR = {2007},
     PAGES = {141--158},
      ISSN = {0046-5755,1572-9168},
   MRCLASS = {20F36 (20F65)},
  MRNUMBER = {2322545},
MRREVIEWER = {Noelle\ C.\ Antony},
       DOI = {10.1007/s10711-007-9148-6},
       URL = {https://doi-org.proxy.binghamton.edu/10.1007/s10711-007-9148-6},
}

@incollection {KO22,
    AUTHOR = {Koberda, Thomas},
     TITLE = {Geometry and combinatorics via right-angled {A}rtin groups},
 BOOKTITLE = {In the tradition of {T}hurston {II}. {G}eometry and groups},
     PAGES = {475--518},
 PUBLISHER = {Springer, Cham},
      YEAR = {[2022] \copyright 2022},
      ISBN = {978-3-030-97559-3; 978-3-030-97560-9},
   MRCLASS = {20F36 (05C45 05C48 05C50 05C60 20F65)},
  MRNUMBER = {4472060},
MRREVIEWER = {Valeriy\ G.\ Bardakov},
       DOI = {10.1007/978-3-030-97560-9\_15},
       URL = {https://doi-org.proxy.binghamton.edu/10.1007/978-3-030-97560-9_15},
}

@article {BLV78,
    AUTHOR = {Bland, Robert G. and Las Vergnas, Michel},
     TITLE = {Orientability of matroids},
   JOURNAL = {J. Combinatorial Theory Ser. B},
  FJOURNAL = {Journal of Combinatorial Theory. Series B},
    VOLUME = {24},
      YEAR = {1978},
    NUMBER = {1},
     PAGES = {94--123},
      ISSN = {0095-8956},
   MRCLASS = {05B35},
  MRNUMBER = {485461},
MRREVIEWER = {Thomas\ Brylawski},
       DOI = {10.1016/0095-8956(78)90080-1},
       URL = {https://doi-org.proxy.binghamton.edu/10.1016/0095-8956(78)90080-1},
}

@article {BB1997,
    AUTHOR = {Bestvina, Mladen and Brady, Noel},
     TITLE = {Morse theory and finiteness properties of groups},
   JOURNAL = {Invent. Math.},
  FJOURNAL = {Inventiones Mathematicae},
    VOLUME = {129},
      YEAR = {1997},
    NUMBER = {3},
     PAGES = {445--470},
      ISSN = {0020-9910},
   MRCLASS = {20F36 (20J05 57M07)},
  MRNUMBER = {1465330},
MRREVIEWER = {John Meier},
       DOI = {10.1007/s002220050168},
       URL = {https://doi-org.ezproxy.library.tufts.edu/10.1007/s002220050168},
}

@article {Cai1997,
    AUTHOR = {Cai, Leizhen},
     TITLE = {On spanning {$2$}-trees in a graph},
   JOURNAL = {Discrete Appl. Math.},
  FJOURNAL = {Discrete Applied Mathematics. The Journal of Combinatorial
              Algorithms, Informatics and Computational Sciences},
    VOLUME = {74},
      YEAR = {1997},
    NUMBER = {3},
     PAGES = {203--216},
      ISSN = {0166-218X},
   MRCLASS = {05C05 (05C85)},
  MRNUMBER = {1444941},
MRREVIEWER = {T. T. Raghunathan},
       DOI = {10.1016/S0166-218X(96)00045-5},
       URL = {https://doi-org.ezproxy.library.tufts.edu/10.1016/S0166-218X(96)00045-5},
}

@article {CRKR25,
    AUTHOR = {Casals-Ruiz, Montserrat and Kazachkov, Ilya and Roy, Mallika},
     TITLE = {Presentation of kernels of rational characters of right-angled
              {A}rtin groups},
   JOURNAL = {Bull. Lond. Math. Soc.},
  FJOURNAL = {Bulletin of the London Mathematical Society},
    VOLUME = {57},
      YEAR = {2025},
    NUMBER = {7},
     PAGES = {2219--2234},
      ISSN = {0024-6093,1469-2120},
   MRCLASS = {20F05 (20E05 20F36 20K15)},
  MRNUMBER = {4936732},
       DOI = {10.1112/blms.70090},
       URL = {https://doi-org.proxy.binghamton.edu/10.1112/blms.70090},
}

@article {DR22,
    AUTHOR = {Deshpande, Priyavrat and Roy, Mallika},
     TITLE = {On the structure of finitely presented {B}estvina-{B}rady
              groups},
   JOURNAL = {Internat. J. Algebra Comput.},
  FJOURNAL = {International Journal of Algebra and Computation},
    VOLUME = {34},
      YEAR = {2024},
    NUMBER = {1},
     PAGES = {69--85},
      ISSN = {0218-1967,1793-6500},
   MRCLASS = {20F36 (08B25 20F65)},
  MRNUMBER = {4716440},
MRREVIEWER = {Valeriy\ G.\ Bardakov},
       DOI = {10.1142/s0218196724500012},
       URL = {https://doi-org.proxy.binghamton.edu/10.1142/s0218196724500012},
}

@article {dromsraag3manifolds,
    AUTHOR = {Droms, Carl},
     TITLE = {Graph groups, coherence, and three-manifolds},
   JOURNAL = {J. Algebra},
  FJOURNAL = {Journal of Algebra},
    VOLUME = {106},
      YEAR = {1987},
    NUMBER = {2},
     PAGES = {484--489},
      ISSN = {0021-8693},
   MRCLASS = {57M15 (05C25 20F05 20F32 57N10)},
  MRNUMBER = {880971},
MRREVIEWER = {R. Z. Goldstein},
       DOI = {10.1016/0021-8693(87)90010-X},
       URL = {https://doi-org.ezproxy.library.tufts.edu/10.1016/0021-8693(87)90010-X},
}

@article {kochloukovamendonontheBNSRsigmainvariantsoftheBBGs,
    AUTHOR = {Kochloukova, Dessislava Hristova and Mendon\c{c}a, Luis},
     TITLE = {On the {B}ieri--{N}eumann--{S}trebel--{R}enz
              {$\Sigma$}-invariants of the {B}estvina--{B}rady groups},
   JOURNAL = {Forum Math.},
  FJOURNAL = {Forum Mathematicum},
    VOLUME = {34},
      YEAR = {2022},
    NUMBER = {3},
     PAGES = {605--626},
      ISSN = {0933-7741},
   MRCLASS = {20J05},
  MRNUMBER = {4415959},
       DOI = {10.1515/forum-2021-0059},
       URL = {https://doi.org/10.1515/forum-2021-0059},
}

@article {DicksLeary99,
    AUTHOR = {Dicks, Warren and Leary, Ian J.},
     TITLE = {Presentations for subgroups of {A}rtin groups},
   JOURNAL = {Proc. Amer. Math. Soc.},
  FJOURNAL = {Proceedings of the American Mathematical Society},
    VOLUME = {127},
      YEAR = {1999},
    NUMBER = {2},
     PAGES = {343--348},
      ISSN = {0002-9939},
   MRCLASS = {20F36 (20F05 57M07)},
  MRNUMBER = {1605948},
MRREVIEWER = {John Meier},
       DOI = {10.1090/S0002-9939-99-04873-X},
       URL = {https://doi.org/10.1090/S0002-9939-99-04873-X},
}

@article {PapadimaSuciuAlgebraicinvariantsforBBGs,
    AUTHOR = {Papadima, Stefan and Suciu, Alexander},
     TITLE = {Algebraic invariants for {B}estvina-{B}rady groups},
   JOURNAL = {J. Lond. Math. Soc. (2)},
  FJOURNAL = {Journal of the London Mathematical Society. Second Series},
    VOLUME = {76},
      YEAR = {2007},
    NUMBER = {2},
     PAGES = {273--292},
      ISSN = {0024-6107},
   MRCLASS = {20F36 (20F14 57M07 57M27)},
  MRNUMBER = {2363416},
       DOI = {10.1112/jlms/jdm045},
       URL = {https://doi.org/10.1112/jlms/jdm045},
}

@article {PapadimaandSuciuBNSRinvariantsandHomologyJumpingLoci,
    AUTHOR = {Papadima, Stefan and Suciu, Alexander I.},
     TITLE = {Bieri-{N}eumann-{S}trebel-{R}enz invariants and homology
              jumping loci},
   JOURNAL = {Proc. Lond. Math. Soc. (3)},
  FJOURNAL = {Proceedings of the London Mathematical Society. Third Series},
    VOLUME = {100},
      YEAR = {2010},
    NUMBER = {3},
     PAGES = {795--834},
      ISSN = {0024-6115},
   MRCLASS = {55N25 (20F65 20J05)},
  MRNUMBER = {2640291},
MRREVIEWER = {Brita E. A. Nucinkis},
       DOI = {10.1112/plms/pdp045},
       URL = {https://doi.org/10.1112/plms/pdp045},






}

@article {LearySaadetogluTheCohomologyofBBGs,
    AUTHOR = {Leary, Ian J. and Saadeto\u{g}lu, M\"{u}ge},
     TITLE = {The cohomology of {B}estvina-{B}rady groups},
   JOURNAL = {Groups Geom. Dyn.},
  FJOURNAL = {Groups, Geometry, and Dynamics},
    VOLUME = {5},
      YEAR = {2011},
    NUMBER = {1},
     PAGES = {121--138},
      ISSN = {1661-7207},
   MRCLASS = {57M07 (20F36)},
  MRNUMBER = {2763781},
MRREVIEWER = {Nicholas A. Koban},
       DOI = {10.4171/GGD/118},
       URL = {https://doi.org/10.4171/GGD/118},
}

@article {lorenzo,
    AUTHOR = {Barquinero, Enrique Miguel and Ruffoni, Lorenzo and Ye, Kaidi},
     TITLE = {Graphical splittings of {A}rtin kernels},
   JOURNAL = {J. Group Theory},
  FJOURNAL = {Journal of Group Theory},
    VOLUME = {24},
      YEAR = {2021},
    NUMBER = {4},
     PAGES = {711--735},
      ISSN = {1433-5883},
   MRCLASS = {20F36 (20F65)},
  MRNUMBER = {4279130},
       DOI = {10.1515/jgth-2020-0124},
       URL = {https://doi-org.ezproxy.library.tufts.edu/10.1515/jgth-2020-0124},
}

@article {DromsIsomorphismsofGraphGroups,
    AUTHOR = {Droms, Carl},
     TITLE = {Isomorphisms of graph groups},
   JOURNAL = {Proc. Amer. Math. Soc.},
  FJOURNAL = {Proceedings of the American Mathematical Society},
    VOLUME = {100},
      YEAR = {1987},
    NUMBER = {3},
     PAGES = {407--408},
      ISSN = {0002-9939},
   MRCLASS = {20F05 (05C25 20F12)},
  MRNUMBER = {891135},
MRREVIEWER = {R. St\"{o}hr},
       DOI = {10.2307/2046419},
       URL = {https://doi-org.ezproxy.wesleyan.edu/10.2307/2046419},
}

@article {ChangJSJofBBGs,
    AUTHOR = {Chang, Yu-Chan},
     TITLE = {Abelian splittings and {JSJ}-decompositions of finitely
              presented {B}estvina-{B}rady groups},
   JOURNAL = {J. Group Theory},
  FJOURNAL = {Journal of Group Theory},
    VOLUME = {26},
      YEAR = {2023},
    NUMBER = {4},
     PAGES = {677--692},
      ISSN = {1433-5883,1435-4446},
   MRCLASS = {20F65 (20E06 20F36)},
  MRNUMBER = {4609850},
MRREVIEWER = {Sam\ Shepherd},
       DOI = {10.1515/jgth-2021-0231},
       URL = {https://doi-org.proxy.lib.ohio-state.edu/10.1515/jgth-2021-0231},
}

@article {DimacaPapadimaSuciuQuasiKahlerBBGs,
    AUTHOR = {Dimca, Alexandru and Papadima, Stefan and Suciu, Alexander I.},
     TITLE = {Quasi-{K}\"{a}hler {B}estvina-{B}rady groups},
   JOURNAL = {J. Algebraic Geom.},
  FJOURNAL = {Journal of Algebraic Geometry},
    VOLUME = {17},
      YEAR = {2008},
    NUMBER = {1},
     PAGES = {185--197},
      ISSN = {1056-3911},
   MRCLASS = {20F65},
  MRNUMBER = {2357684},
MRREVIEWER = {Eddy Godelle},
       DOI = {10.1090/S1056-3911-07-00463-8},
       URL = {https://doi.org/10.1090/S1056-3911-07-00463-8},
}

@article {Dirac1961,
    AUTHOR = {Dirac, G. A.},
     TITLE = {On rigid circuit graphs},
   JOURNAL = {Abh. Math. Sem. Univ. Hamburg},
  FJOURNAL = {Abhandlungen aus dem Mathematischen Seminar der Universit\"at
              Hamburg},
    VOLUME = {25},
      YEAR = {1961},
     PAGES = {71--76},
      ISSN = {0025-5858,1865-8784},
   MRCLASS = {05.40},
  MRNUMBER = {130190},
MRREVIEWER = {F.\ Harary},
       DOI = {10.1007/BF02992776},
       URL = {https://doi.org/10.1007/BF02992776},
}

@article {Bridson2020,
    AUTHOR = {Bridson, Martin R.},
     TITLE = {On the recognition of right-angled {A}rtin groups},
   JOURNAL = {Glasg. Math. J.},
  FJOURNAL = {Glasgow Mathematical Journal},
    VOLUME = {62},
      YEAR = {2020},
    NUMBER = {2},
     PAGES = {473--475},
      ISSN = {0017-0895,1469-509X},
   MRCLASS = {20F36 (20F10)},
  MRNUMBER = {4085052},
MRREVIEWER = {Mohammad\ N.\ Abdulrahim},
       DOI = {10.1017/s0017089519000235},
       URL = {https://doi.org/10.1017/s0017089519000235},
}

@article {GO96,
    AUTHOR = {Gutierrez, M. and Oubi\~na, L.},
     TITLE = {Metric characterizations of proper interval graphs and
              tree-clique graphs},
   JOURNAL = {J. Graph Theory},
  FJOURNAL = {Journal of Graph Theory},
    VOLUME = {21},
      YEAR = {1996},
    NUMBER = {2},
     PAGES = {199--205},
      ISSN = {0364-9024,1097-0118},
   MRCLASS = {05C75 (05C05)},
  MRNUMBER = {1368745},
       DOI = {10.1002/(SICI)1097-0118(199602)21:2<199::AID-JGT9>3.0.CO;2-M},
       URL =
              {https://doi-org.proxy.binghamton.edu/10.1002/(SICI)1097-0118(199602)21:2<199::AID-JGT9>3.0.CO;2-M},
}

@article {BDCV98,
    AUTHOR = {Brandst\"adt, Andreas and Dragan, Feodor and Chepoi, Victor
              and Voloshin, Vitaly},
     TITLE = {Dually chordal graphs},
   JOURNAL = {SIAM J. Discrete Math.},
  FJOURNAL = {SIAM Journal on Discrete Mathematics},
    VOLUME = {11},
      YEAR = {1998},
    NUMBER = {3},
     PAGES = {437--455},
      ISSN = {0895-4801,1095-7146},
   MRCLASS = {05C75 (05C65 68R10)},
  MRNUMBER = {1628114},
MRREVIEWER = {Valentin\ E.\ Brimkov},
       DOI = {10.1137/S0895480193253415},
       URL = {https://doi-org.proxy.binghamton.edu/10.1137/S0895480193253415},
}

@article {LCC07,
    AUTHOR = {Lin, Ching-Chi and Chang, Gerard J. and Chen, Gen-Huey},
     TITLE = {Locally connected spanning trees in strongly chordal graphs
              and proper circular-arc graphs},
   JOURNAL = {Discrete Math.},
  FJOURNAL = {Discrete Mathematics},
    VOLUME = {307},
      YEAR = {2007},
    NUMBER = {2},
     PAGES = {208--215},
      ISSN = {0012-365X,1872-681X},
   MRCLASS = {05C62 (05C05 05C85 68Q25)},
  MRNUMBER = {2285191},
MRREVIEWER = {Terry\ A.\ McKee},
       DOI = {10.1016/j.disc.2006.06.026},
       URL = {https://doi-org.proxy.binghamton.edu/10.1016/j.disc.2006.06.026},
}

@article {SB94,
    AUTHOR = {Szwarcfiter, Jayme L. and Bornstein, Claudson F.},
     TITLE = {Clique graphs of chordal and path graphs},
   JOURNAL = {SIAM J. Discrete Math.},
  FJOURNAL = {SIAM Journal on Discrete Mathematics},
    VOLUME = {7},
      YEAR = {1994},
    NUMBER = {2},
     PAGES = {331--336},
      ISSN = {0895-4801},
   MRCLASS = {05C05 (05C12 05C85)},
  MRNUMBER = {1272006},
MRREVIEWER = {Ko-Wei\ Lih},
       DOI = {10.1137/S0895480191223191},
       URL = {https://doi-org.proxy.binghamton.edu/10.1137/S0895480191223191},
}

@article {BL25,
    AUTHOR = {Blumer, S.},
     TITLE = {Subgroups of {B}estvina-{B}rady groups},
   JOURNAL = {J. Pure Appl. Algebra},
  FJOURNAL = {Journal of Pure and Applied Algebra},
    VOLUME = {229},
      YEAR = {2025},
    NUMBER = {10},
     PAGES = {Paper No. 108080, 18},
      ISSN = {0022-4049,1873-1376},
   MRCLASS = {20F65 (12F12 17B70 20F36 20J06)},
  MRNUMBER = {4954383},
       DOI = {10.1016/j.jpaa.2025.108080},
       URL = {https://doi.org/10.1016/j.jpaa.2025.108080},
}

@article {duallychordal2012,
    AUTHOR = {De Caria, Pablo and Gutierrez, Marisa},
     TITLE = {On minimal vertex separators of dually chordal graphs:
              properties and characterizations},
   JOURNAL = {Discrete Appl. Math.},
  FJOURNAL = {Discrete Applied Mathematics. The Journal of Combinatorial
              Algorithms, Informatics and Computational Sciences},
    VOLUME = {160},
      YEAR = {2012},
    NUMBER = {18},
     PAGES = {2627--2635},
      ISSN = {0166-218X,1872-6771},
   MRCLASS = {05C40 (05C75)},
  MRNUMBER = {2971345},
       DOI = {10.1016/j.dam.2012.02.022},
       URL = {https://doi-org.proxy.lib.ohio-state.edu/10.1016/j.dam.2012.02.022},
}

@book {Oxley,
    AUTHOR = {Oxley, James},
     TITLE = {Matroid theory},
    SERIES = {Oxford Graduate Texts in Mathematics},
    VOLUME = {21},
   EDITION = {Second},
 PUBLISHER = {Oxford University Press, Oxford},
      YEAR = {2011},
     PAGES = {xiv+684},
      ISBN = {978-0-19-960339-8},
   MRCLASS = {05-01 (05B35 90C27)},
  MRNUMBER = {2849819},
MRREVIEWER = {Maruti\ M.\ Shikare},
       DOI = {10.1093/acprof:oso/9780198566946.001.0001},
       URL = {https://doi-org.proxy.lib.ohio-state.edu/10.1093/acprof:oso/9780198566946.001.0001},
}

@book {orientedmatroids,
    AUTHOR = {Bj\"orner, Anders and Las Vergnas, Michel and Sturmfels, Bernd
              and White, Neil and Ziegler, G\"unter M.},
     TITLE = {Oriented matroids},
    SERIES = {Encyclopedia of Mathematics and its Applications},
    VOLUME = {46},
   EDITION = {Second},
 PUBLISHER = {Cambridge University Press, Cambridge},
      YEAR = {1999},
     PAGES = {xii+548},
      ISBN = {0-521-77750-X},
   MRCLASS = {52B40 (05B35 52C35)},
  MRNUMBER = {1744046},
       DOI = {10.1017/CBO9780511586507},
       URL = {https://doi-org.proxy.lib.ohio-state.edu/10.1017/CBO9780511586507},
}

@article {oxleycrenshaw,
    AUTHOR = {Crenshaw, Cameron and Oxley, James},
     TITLE = {Ordering circuits of matroids},
   JOURNAL = {Electron. J. Combin.},
  FJOURNAL = {Electronic Journal of Combinatorics},
    VOLUME = {29},
      YEAR = {2022},
    NUMBER = {4},
     PAGES = {Paper No. 4.31, 28},
      ISSN = {1077-8926},
   MRCLASS = {05B35},
  MRNUMBER = {4511331},
MRREVIEWER = {Laura\ Bertani},
       DOI = {10.37236/11117},
       URL = {https://doi-org.proxy.lib.ohio-state.edu/10.37236/11117},
}

@book {Laura2025,
    AUTHOR = {Anderson, Laura},
     TITLE = {Oriented matroids},
    SERIES = {Cambridge Studies in Advanced Mathematics},
    VOLUME = {216},
 PUBLISHER = {Cambridge University Press, Cambridge},
      YEAR = {2025},
     PAGES = {xii+321},
      ISBN = {9-781-009-49411-3; [9781009494076]},
   MRCLASS = {52C40 (05B35)},
  MRNUMBER = {4880415},
}

@article {DehnFunc26,
    AUTHOR = {Chang, Yu-Chan and Garc\'ia-Mej\'ia, Jer\'onimo and
              Migliorini, Matteo},
     TITLE = {Complete classification of the {D}ehn functions of
              {B}estvina-{B}rady groups},
   JOURNAL = {Geom. Funct. Anal.},
  FJOURNAL = {Geometric and Functional Analysis},
    VOLUME = {36},
      YEAR = {2026},
    NUMBER = {1},
     PAGES = {1--58},
      ISSN = {1016-443X,1420-8970},
   MRCLASS = {20F69 (20F05 20F38 20F65 51F30)},
  MRNUMBER = {5024899},
       DOI = {10.1007/s00039-026-00731-7},
       URL = {https://doi-org.proxy.lib.ohio-state.edu/10.1007/s00039-026-00731-7},
}

@article {SplittingBBG25,
    AUTHOR = {Chang, Yu-Chan},
     TITLE = {A note on the splittings of finitely presented
              {B}estvina-{B}rady groups},
   JOURNAL = {Glasg. Math. J.},
  FJOURNAL = {Glasgow Mathematical Journal},
    VOLUME = {67},
      YEAR = {2025},
    NUMBER = {2},
     PAGES = {224--227},
      ISSN = {0017-0895,1469-509X},
   MRCLASS = {20F65 (20E06 20E08)},
  MRNUMBER = {4883972},
MRREVIEWER = {Jens\ Harlander},
       DOI = {10.1017/S0017089524000338},
       URL = {https://doi-org.proxy.lib.ohio-state.edu/10.1017/S0017089524000338},
}

@misc{jialin,
      title={Automorphisms of Bestvina-Brady Groups: IA Rigidity, Arithmetic Commensurability, and Finiteness}, 
      author={Jialin Lei},
      year={2026},
      eprint={2607.23380},
      archivePrefix={arXiv},
      primaryClass={math.GR},
      url={https://arxiv.org/abs/2607.23380}, 
      note={(Revised version in preparation)}
}

@article {DavisOkun,
    AUTHOR = {Davis, Michael W. and Okun, Boris},
     TITLE = {Cohomology computations for {A}rtin groups, {B}estvina-{B}rady
              groups, and graph products},
   JOURNAL = {Groups Geom. Dyn.},
  FJOURNAL = {Groups, Geometry, and Dynamics},
    VOLUME = {6},
      YEAR = {2012},
    NUMBER = {3},
     PAGES = {485--531},
      ISSN = {1661-7207,1661-7215},
   MRCLASS = {20J06 (20F36 20F55)},
  MRNUMBER = {2961283},
MRREVIEWER = {Aditi\ Kar},
       DOI = {10.4171/GGD/164},
       URL = {https://doi-org.proxy.lib.ohio-state.edu/10.4171/GGD/164},
}

\end{document}